\documentclass[12pt,a4paper]{amsart}
\usepackage{geometry}
\usepackage[english]{babel}
\usepackage{amssymb,enumerate,latexsym,hyperref,color,graphicx,cite}

\newenvironment{proof*}[1]{\medskip\noindent\textbf{#1\ }}{\hspace*{\fill}$\Box$\medskip}
\newtheorem{theorem}{Theorem}[section]
\newtheorem{lemma}[theorem]{Lemma}
\newtheorem{proposition}[theorem]{Proposition}
\newtheorem{corollary}[theorem]{Corollary}
\newtheorem{definition}[theorem]{Definition}
\theoremstyle{remark}
\newtheorem{remark}[theorem]{Remark}
\newtheorem{example}[theorem]{Example}
\newtheorem{problem}[theorem]{Problem}
\newtheorem{conjecture}[theorem]{Conjecture}

\newcommand{\PP}{{\mathbb P}}

\newcommand{\RR}{{\mathbb R}}

\newcommand{\TT}{{\mathbb T}}

\newcommand{\ZZ}{{\mathbb Z}}

\newcommand{\be}[1]{\begin{equation} \label{#1} }
\newcommand{\ee}{\end{equation}}
\newcommand{\beq}{\begin{equation}}

\def \Diff{{\rm Diff}}

\def \hx0{\hat{x_0}}

\begin{document}

\title[Fractal vs Regularity and Rigidity]{Invariant Distributions of Partially Hyperbolic Systems: Fractal Graphs, Excessive Regularity, and Rigidity}
\author{disheng xu}
\address{School of Science, Great Bay University and Great bay institute for advanced study, 
Songshan Lake International Innovation Entrepreneurship Community A5, Dongguan 523000, CHINA}
\email{xudisheng@gbu.edu.cn}

\author{jiesong zhang}
\address{School of Mathematical Sciences, Peking University, No.5 Yiheyuan Road, Haidian District, Beijing 100871, China}
\email{zhjs@stu.pku.edu.cn}

\begin{abstract}
We introduce a novel approach linking fractal geometry to partially hyperbolic dynamics, revealing several new phenomena related to regularity jumps and rigidity. One key result demonstrates a sharp phase transition for partially hyperbolic diffeomorphisms $f$ with a contracting center direction: $f \in \Diff^\infty_{\mathrm{vol}}(\TT^3)$ is $C^\infty$-rigid if and only if both $E^s$ and $E^c$ exhibit Hölder exponents exceeding the expected threshold. Moreover, for $f \in \Diff^2_{\mathrm{vol}}(\TT^3)$, we prove:  
\begin{itemize}
    \item If the Hölder exponent of $E^s$ exceeds the expected value, then $E^s$ is $C^{1}$ and $E^u \oplus E^s$ is jointly integrable.
    \item If the Hölder exponent of $E^c$ exceeds the expected value, then $W^c$ forms a $C^{1}$ foliation.
    \item If $E^s$ (or $E^c$) does not exhibit excessive Hölder regularity, it must have a \textit{fractal graph}.
\end{itemize}  
These and related results originate from a general non-fractal invariance principle: for a skew product $F$ over a partially hyperbolic system $f$, if $F$ expands fibers more weakly than $f$ along $W^u_f$ in the base, then for any $F$-invariant section $\Phi$, if $\Phi$ has no fractal graph, then it is smooth along $W^u_f$ and holonomy-invariant.  

Motivated by these findings, we propose a new conjecture on the \textit{stable fractal or stable smooth} behavior of invariant distributions in typical partially hyperbolic diffeomorphisms.

\end{abstract}

\maketitle
\tableofcontents

\section{Introduction}
\subsection{Partially hyperbolic systems and associated invariant distributions}

Diffeomorphisms exhibiting hyperbolic behavior form a fruitful part of dynamical systems theory. The strongest form of such behavior are known as uniformly hyperbolicity or Anosov systems. 

Partially hyperbolic systems are a natural generalization of Anosov systems, allowing for the inclusion of neutral directions. A diffeomorphism \( f : M \to M \) is called partially hyperbolic if there exists a Riemannian metric and a continuous, \( Df \)-invariant splitting of \( TM \) as 
\[ 
TM = E^s \oplus E^c \oplus E^u, 
\] 
such that, for any \( x \in M \),
\[
\|Df^k|_{E^s(x)}\| < 1 < \|(Df^k|_{E^u(x)})^{-1}\|^{-1}
\]
and
\[
\|Df^k|_{E^s(x)}\| < \|(Df^k|_{E^c(x)})^{-1}\|^{-1}, \quad \|Df^k|_{E^c(x)}\| < \|(Df^k|_{E^u(x)})^{-1}\|^{-1}.
\]
In particular, a partially hyperbolic diffeomorphism is Anosov if \( E^c \) is trivial. This broader class of systems is more flexible, permitting a wider range of dynamical behaviors, and leads to a variety of new phenomena not observed in uniformly hyperbolic systems. Examples of partially hyperbolic systems include time-\( t \) maps of geodesic flows on negatively curved compact Riemannian manifolds, suspensions of Anosov systems, skew products over Anosov systems, and others. Like uniform hyperbolicity, partial hyperbolicity is robust under \( C^1 \)-perturbations.

A notable feature of the invariant distributions in partially hyperbolic systems is their H\"older regularity. An important example is Anosov systems, where the stable and unstable distributions \( E^{s,u} \) and their associated foliations \( W^{s,u} \) are both Hölder continuous. Historically, the ergodicity problem for volume-preserving Anosov systems (or flows) was a significant challenge. Anosov \cite{Ano69} generalized Hopf's argument \cite{hopf} to solve this problem, which relies on the absolute continuity of \( W^{s,u} \).  Meanwhile, the proof of absolute continuity for \( W^{s,u} \) in Anosov systems depends on distortion estimates of certain Jacobians, which are themselves derived from the Hölder continuity of \( E^{s,u} \) .

In this paper, we investigate the properties of invariant distributions in partially hyperbolic systems. 


\subsection{Low regularity of invariant distributions vs bunching property vs rigidity phenomena}
A key property of partially systems is that \( E^s \) and \( E^u \) are uniquely integrable, and their corresponding integral foliations \( W^s \) and \( W^u \) are called the stable and unstable foliations. The study of the regularity of these distributions are closely related to the regularity of the corresponding integral foliations \cite{psw97,hw99}. However, the center distribution \( E^c \) may fail to be integrable (see \cite{bw08} and references therein). Furthermore, \( E^s \oplus E^u \) is generally not integrable, regardless of the regularity of \( E^s \oplus E^u \).

These dynamically defined distributions and foliations usually do not exhibit excessive regularity. In fact, for any $\alpha \in (0,1)$, there exists open set of Anosov diffeomorphisms where the H\"older exponent of their stable and unstable distributions \( E^{s,u} \) and their associated foliations \( W^{s,u} \) are less than $\alpha$ \cite{has95,hw99}. 

Under certain additional \textit{bunching} conditions, the distributions and foliations defined by dynamical systems may exhibit \textit{robust} \( C^1 \) smoothness \cite{hps77,has95,psw97}. Without bunching conditions, additional regularity is rare and usually appears with certain rigidity phenomena. For instance, 
\begin{itemize}
    \item  For a partially hyperbolic system smoothly conjugate to a linear system, the corresponding bundles are smooth (\textit{smooth rigidity}).
    \item  For a \( C^2 \) partially hyperbolic system \( f \) homotopic to an Anosov system on \( \mathbb{T}^3 \), if \( E^s \oplus E^u \) is jointly integrable, i.e., if there exists a foliation \( W^{su} \) tangent to \( E^s \oplus E^u \) everywhere, then \( f \) is an Anosov system  \cite{gs20,hs21}. In this case, both \( E^s \) and \( E^u \) are \( C^1 \) distributions (\textit{topological rigidity}). 
    \item For a volume preserving perturbation of a geodesic flow on closed negatively curved surfaces or volume preserving partially hyperbolic diffeomorphisms with compact centers on $\TT^3$, if $E^c$ is absolutely continuous, then $E^c$ is $C^\infty$ \cite{avw1,avw2} (\textit{measure rigidity}). 
\end{itemize}
It is important to note that these rigidity phenomena are relatively rare within the ``moduli space'' of the corresponding systems and are \textit{fragile} under small perturbations.



All discussions above lead to the following natural problem:

\begin{problem}\label{que: main}
When do invariant distributions of partially hyperbolic systems exhibit only Hölder regularity? What can we say about these systems with distributions of low regularity? 
\end{problem}

In this paper, partially inspired by the study of fractal geometry, we offer a potential answer to Problem \ref{que: main}: \textbf{invariant distributions of general partially hyperbolic systems are only Hölder continuous only when they possess fractal graphs.}

\begin{figure}[ht]
  \centering
  \includegraphics[width=0.7\textwidth]{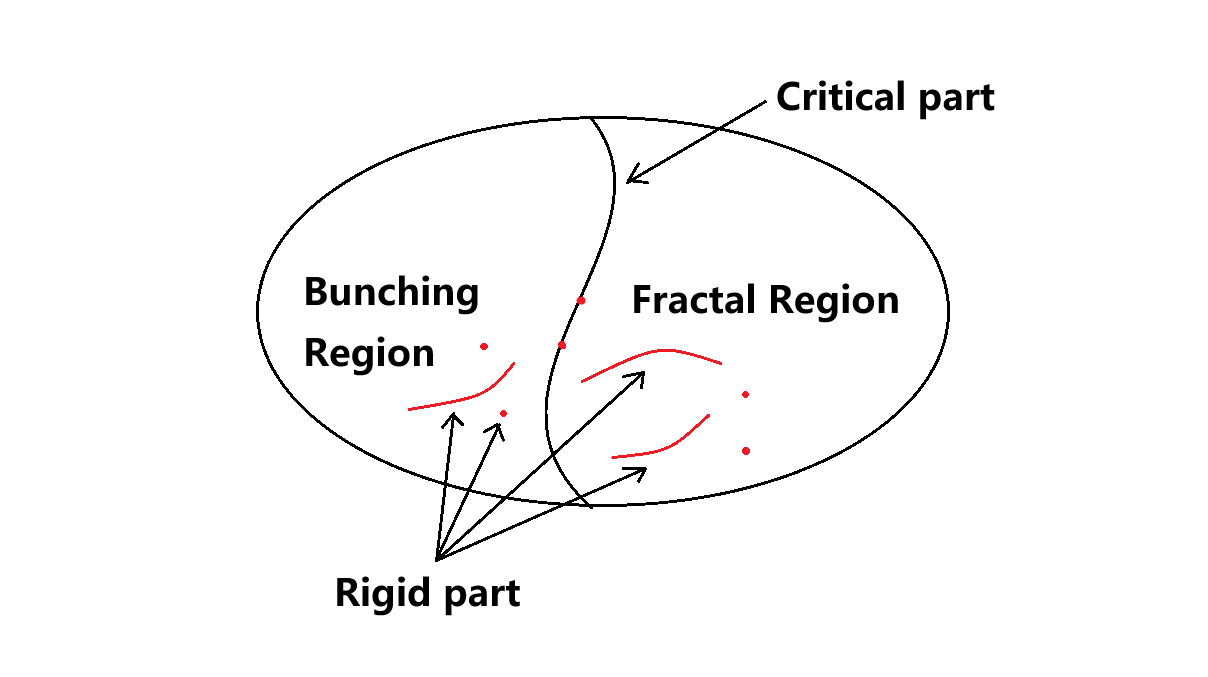}
  \caption{The space of smooth partially hyperbolic systems: the bunching region is open, the fractal region is conjectured to be open, and the critical part and rigid part are conjectured to have no interior.} 
\end{figure}

To present this answer more rigorously, we introduce the definition of \textit{fractal graphs}. 
\begin{remark} It is worth noting that there are examples neither in the bunching region nor the fractal region: the strong stable (unstable) distribution for volume-preserving Anosov flows on three-manifolds is $\alpha$-H\"older for any $\alpha < 1$ but generally fails to be Lipschitz \cite{fh03}. And it is corresponding to the \textit{critical part} in Figure 1.
\end{remark}

\subsection{Fractal geometry and dimension theory of graphs (or invariant sections)}

In fractal geometry, the size of fractal sets is typically characterized by various types of dimensions. For example, the Hausdorff dimension, denoted $\dim_H$, of a set is generally less than or equal to its lower box dimension $\underline{\dim}_B$ and packing dimension $\dim_P$; both of these are also less than or equal to the upper box dimension $\overline{\dim}_B$ of the set (see \cite{fal90}, for instance).

A central topic in fractal geometry is the dimension of the graph of a function. The graphs of Weierstrass-type and related functions are among the most studied objects in this area. Recently, a beautiful dichotomy result was established in \cite{rs21} concerning the dimension of the graph of Weierstrass-type functions
\[
W^\phi_{\lambda,b}(x) = \sum_{n=1}^\infty \lambda^n \phi(b^n(x)), \quad b \in \{2,3,\dots\}, \quad \lambda \in \left(\frac{1}{b}, 1\right), \quad \phi \in C^k(\mathbb{R}/\mathbb{Z}), \quad k \in \{5,\dots, \infty, \omega\},
\]
where either the graph of \( W^\phi_{\lambda,b} \) has the predicted Hausdorff dimension (hence is fractal), or is \( C^k \). In the case where \( \phi \) is a trigonometric function, the box dimension of \( W \) was proven in \cite{kap84}, and the Hausdorff dimension of \( W \) was established in \cite{shen18}, resolving a long-standing conjecture.

A natural generalization of the graph of a function is the section of a fiber bundle. For example, the graph \( \{ x, f(x) \} \subset X \times \mathbb{R} \) of a function \( f: X \to \mathbb{R} \) can be viewed as a section of the trivial bundle \( \pi: X \times \mathbb{R} \to X \). We now introduce the definition of a \textit{fractal graph} for general sections of fiber bundles. Let \( \pi: N \to M \) be a non-trivial \( C^1 \) fiber bundle and \( \Phi: M \to N \) be a continuous section, i.e., \( \pi \circ \Phi = \text{id}_M \). Then, we define the dimension of the graph of \( \Phi \) as
\[
\dim(\mathrm{Graph}(\Phi)) := \dim(\{ \Phi(x) \in N \}_{x \in M}),
\]
where \( \dim \) can refer to any of the dimensions \( \dim_H \), \( \underline{\dim}_B \), \( \dim_P \), or \( \overline{\dim}_B \). Furthermore, if
\[
\dim(\mathrm{Graph}(\Phi)) - \dim(M) > 0 \quad (\text{hence } \overline{\dim}_B(\mathrm{Graph}(\Phi)) - \dim(M) > 0),
\]
then the Hölder exponent of \( \Phi \) satisfies
\[
\text{the Hölder exponent of } \Phi \leq 1 - \frac{\dim(\mathrm{Graph}(\Phi)) - \dim(M)}{\dim N - \dim M}. \footnote{\text{See Section \ref{sec kdim} for a proof.}}
\]
In other words, a function, or more generally a section, cannot exhibit Hölder regularity higher than what is predicted by the dimension of its graph. This motivates the following definition.
\begin{definition}\label{def: frct grph}
We say that a continuous section \( \Phi: M \to N \) of a \( C^1 \) fiber bundle \( \pi: N \to M \) has a \textbf{fractal graph} if 
\[
\overline{\dim}_B (\mathrm{Graph}(\Phi)) > \dim M.
\]
\end{definition}
\subsection{Non-fractal invariance principle for invariant sections}\label{subsec: gene dich}
Some key tools for estimating the regularity of invariant distributions from below are the \( C^r \) section theorem and the H\"older section theorem \cite{hps77,psw10}. Roughly speaking, the strategy involves considering a bundle map $F:N \to N$ over a partially hyperbolic diffeomorphism $f:M \to M$. If the expanding rate of $F$ along the fibers are bounded from below in some fashion, then the \( C^r \) section theorem and the H\"older section theorem give low bound of regularity of the $F$-invariant section from below. Then by associating invariant distributions with invariant sections of certain bundle maps, it leads to regularity estimates for specific invariant sections.

Our main result is somehow opposite to the H\"older section theorem: if the expanding rate of $F$ along the fibers are bounded from above, then the $F$-invariant section has a fractal graph, and thus a upper bound of regularity, unless certain rigidity phenomenon occurs. Specifically, consider a $C^{1}$ partially hyperbolic system \( f \) on a compact manifold \( M \) and a \( C^{1} \) map \( F \) on a fiber bundle \( N \) over \( M \), such that \( F \) projects to \( f \). If \( F \) expands fibers more weakly than \( f \) does along the unstable foliation \( W^u_f \), then \( F \) admits a canonically defined holonomy \( h^u_F \) between fibers. For any $F$ invariant section $\Phi$, we have the following alternative:
\begin{itemize}
\item \textbf{(Fractal graph)}: \textit{\( \Phi \) has a fractal graph}.
\item \textbf{(Invariance principle)}: $\Phi$ is $h^u_F$-invariant and $C^{1}$ along $W^u_f$. 
\end{itemize}
This general result leads to several applications in the theory of partially hyperbolic systems, which we shall explore in the following sections. To the best of the author's knowledge, these results are the first to relate fractal properties to regularity estimate and rigidity in the theory of partially hyperbolic systems.





\subsection{Applications to partially hyperbolic Anosov systems: regularity bootstrap} \label{section 3dim}

To demonstrate the power of our result, we present applications in the study of regularity bootstrap of the invariant distributions, i.e., moderately high regularity implies higher regularity. This phenomenon happens for Anosov splittings under various conditions \cite{hasboot,has02,kh90,ghy93,fk91}. By analyzing the fractal properties of invariant distributions, we discover such bootstrapping phenomenon in partially hyperbolic systems. To simplify the notations, we first focus on \textit{partially hyperbolic Anosov systems} on $\TT^3$, defined as follows (see Section \ref{sec boot high dim} for more general results). 


\begin{definition}
A partially hyperbolic diffeomorphism \( f \) is called \textit{partially hyperbolic Anosov} if \( Df \) uniformly contracts (or expands) the center distribution \( E^c \).
\end{definition}

Partially hyperbolic Anosov diffeomorphisms form an important class of partially hyperbolic systems and our understanding of these systems is still developing. For instance, it was only recently shown by Avila-Crovisier-Eskin-Potrie-Wilkinson-Zhang that the stable foliation \( W^s \) is minimal for any \( C^{1+\alpha} \) partially hyperbolic Anosov diffeomorphism (with a contracting center) on \( \mathbb{T}^3 \). We now fix the following notations.

\textit{Notations:}
\begin{itemize}
    \item For any \( \beta \in \mathbb{R}^+ \), we say a map \( \varphi \) is \( C^{\beta+} \) if it is \( C^{\beta+\epsilon} \) for some \( \epsilon > 0 \), and \( \varphi \) is \( C^{\beta-} \) if it is \( C^{\beta-\epsilon} \) for all \( \epsilon > 0 \). Similarly, if \( \beta \in (0,1) \), we say a map \( \varphi \) is \( \beta+ \)-Hölder continuous if it is \( \beta + \epsilon \)-Hölder continuous for some \( \epsilon > 0 \), and \( \varphi \) is \( \beta- \)-Hölder continuous if it is \( \beta - \epsilon \)-Hölder continuous for all \( \epsilon > 0 \).
    \item In the rest of the paper, when we refer to \( f \) as a partially hyperbolic Anosov diffeomorphism, we assume that \( f \) uniformly contracts \( E^c \) (otherwise, we consider \( f^{-1} \) instead).
\end{itemize}

\subsubsection*{Dichotomy Results for \( E^s \):}

Our first result is a \textit{fractal graph vs. smoothness and rigidity} dichotomy for \( E^s \) in conservative partially hyperbolic Anosov diffeomorphisms on \( \mathbb{T}^3 \). Let
\[
\theta_s = \theta_s(f) := \sup_\theta \left\{\theta : \exists k \text{ such that } |Df^k|_{E^s(x)}| \cdot |Df^k|_{E^u(x)}|^\theta < |Df^k|_{E^c(x)}|, \forall x \in \mathbb{T}^3 \right\}.
\]
It is clear that \( \theta_s \) is closely related to pinching assumptions in the theory of smooth dynamical systems. By standard methods, one can show that \textbf{\( E^s \) is \( \theta_s-\)-Hölder continuous} (see Corollary \ref{coro pinching} for a proof). The following theorem shows that this bound is (almost) sharp.

\begin{theorem} \label{theorem 3dim s}
Suppose \( f \in \Diff^{2}_{\mathrm{vol}}(\mathbb{T}^3) \) is a partially hyperbolic Anosov diffeomorphism. Then:
\begin{enumerate}
    \item If \( E^s \) is \( \theta_s +\)-Hölder continuous, then \( E^s \) is actually \( C^{1} \). Furthermore, \( E^s \oplus E^u \) is jointly integrable.
    \item If (1) does not hold, then \( E^s \) \footnote{As a continuous section of the Grassmannian bundle $G^1\TT^3$ over $\TT^3$.} has a fractal graph.
\end{enumerate}
\end{theorem}

\begin{remark}
\begin{enumerate}[(1)]
    \item Theorem \ref{theorem 3dim s} extends to more general cases, including a higher-dimensional generalization in Theorem \ref{theorem s bunching}. It also holds for partially hyperbolic Anosov systems on \( \mathbb{T}^3 \) that are \( C^1 \)-close to being volume-preserving, such as those \textbf{\( C^1 \)-close to a linear system}. A recent result  \cite{alos24} proves measure rigidity for Gibbs measures for this class of systems, assuming \( E^s \oplus E^u \) is not jointly integrable.
    \item Since \( E^s \) is always \( \theta_s-\)-Hölder continuous, our result provides an (almost) sharp bootstrap result for the Hölder regularity of \( E^s \).
    \item We also show that the lower box dimension \( \underline{\dim}_B(E^s) > \dim(\mathbb{T}^3) \), which is a stronger condition than the fractal graph property defined in Definition \ref{def: frct grph}.
\end{enumerate}
\end{remark}

An immediate corollary of Theorem \ref{theorem 3dim s} is the following dichotomy result.

\begin{corollary} \label{coro: local qualitative dich}
For any partially hyperbolic Anosov \( f \in \Diff_{\mathrm{vol}}^{2}(\mathbb{T}^3) \), \( E^s \) is either \( C^{1} \) or has a fractal graph. In particular, for those \( f \in \Diff_{\mathrm{vol}}^{2}(\mathbb{T}^3) \) that are not in a codimension-\( \infty \) subset in $\Diff_{\mathrm{vol}}^{2}(\mathbb{T}^3)$, \( E^s_f \) has a fractal graph.
\end{corollary}

\subsubsection*{Alternative Results for \( E^c \):}

Next, we consider a similar alternative results for the regularity of \( E^c \) along the stable foliation \( W^s \) and the regularity of \( W^c \). Let
\[
\theta_c = \theta_c(f) := \sup_\theta \left\{\theta : \exists k \text{ such that } |Df^k|_{E^s(x)}|^{1-\theta} < |Df^k|_{E^c(x)}|, \forall x \in \mathbb{T}^3 \right\}.
\]
It is not difficult to show that for partially hyperbolic Anosov \( f \) on \( \mathbb{T}^3 \), \( E^c \) is \( \theta_c-\)-Hölder continuous along \( W^s \) (in fact, globally \( \theta_c-\)-Hölder continuous, see Corollary \ref{coro pinching}).

\begin{theorem} \label{theorem 3dim c}
Suppose \( f \in \Diff^{2}(\mathbb{T}^3) \) is a partially hyperbolic Anosov diffeomorphism with minimal \( W^s \). Then:
\begin{enumerate}
    \item If \( E^c \) is uniformly \( \theta_c+\)-Hölder continuous along \( W^s \), then \( E^c \) is uniformly \( C^{1} \) along \( W^s \), and consequently, \( W^c \) is a \( C^{1} \) foliation.
    \item If (1) does not hold, then \( E^c \) has a fractal graph.
\end{enumerate}
\end{theorem}

\begin{remark}
\begin{enumerate}[(1)]
    \item For the results on \( E^c \), no volume-preserving assumption is needed.
    \item If we assume the result of Avila-Crovisier-Eskin-Potrie-Wilkinson-Zhang holds, we can drop the assumption that \( W^s \) is minimal.
    \item Like for \( E^s \), we obtain an (almost) sharp bootstrap result for the Hölder regularity of \( E^c \).
\end{enumerate}
\end{remark}

Unlike \( W^s \) and \( W^u \), the center foliation \( W^c \) may fail to be absolutely continuous, if it exists. This phenomenon, often referred to as \textit{pathological foliation}, has been observed in several partially hyperbolic systems (see \cite{mil97, sw00, gog12, sx09, ph07}), though it is still not fully understood, especially when \( W^c \) is non-compact. It has been conjectured that for typical partially hyperbolic systems, the center foliation is not absolutely continuous (\cite{hertznote} Problem 12). The following theorem provides a necessary condition for a center foliation to be pathological:

\begin{corollary} \label{path imply fractal}
Let \( f \) be a diffeomorphism as in Theorem \ref{theorem 3dim c}. If \( W^c \) is pathological, then \( E^c \) has a fractal graph.
\end{corollary}
It would be very interesting to further investigate the relationship between the fractal property of \( E^c \) and the pathological behavior of \( W^c \). In particular, under the framework of Theorem \ref{theorem 3dim c}, one might wonder whether it is possible for \( E^c \) to have a fractal graph while \( W^c \) is absolutely continuous. A priori, this could be the case, since a $C^{1}$ foliation could have a distribution with a fractal graph \footnote{For example, let $w(x):\TT^1 \to \TT^1$ be a continuous function with fractal graph, and consider the foliation $W$ of $\TT^2$ generated by the vector field $E(x,y):=(1,w(x))$. }.

For similar reasons, the following local dichotomy result is not a direct corollary of Theorem \ref{theorem 3dim c}.

\begin{theorem} \label{theorem 3dim c local}
For any \( f \in \Diff^{2}(\mathbb{T}^3) \) that is \( C^1 \)-close to a partially hyperbolic Anosov automorphism on \( \mathbb{T}^3 \) such that \( W^s_f \) is minimal, \( E^c_f \) is either \( C^{1} \) or has a fractal graph.
\end{theorem}

\subsection{Applications to partially hyperbolic Anosov systems: global rigidity}

Recall that for an Anosov diffeomorphism \( f \) on a compact manifold \( M \), we say \( f \) is \( C^r \)-rigid if it is \( C^r \)-conjugate to an automorphism of an infra-nilmanifold. It is well known that \( f \) is always \( C^0 \)-rigid if \( M \) is an infra-nilmanifold \cite{fra69,man74}, but in general, \( f \) is not necessarily \( C^1 \)-rigid. To obtain a \( C^{1+} \)-rigidity result, additional assumptions are required. For example, under assumptions on periodic data \cite{lmmt2,lla92,gg08,gog08,ks09,gks11,dg24} and Lyapunov exponents \cite{sy19,dew21,gks20}, one expects \( f \) to be \( C^{1} \)-rigid, or even \( C^\infty \)-rigid if \( f \) is \( C^\infty \) \cite{gog17,ksw23,ksw24}.

A key assumption closely related to \( C^{1+} \)-rigidity is the regularity of the invariant distributions. The first such result was shown in \cite{av68}, where it was proved that an Anosov diffeomorphism on \( \mathbb{T}^2 \) with \( C^\infty \) stable and unstable distributions is \( C^\infty \)-conjugate to a linear automorphism. Later, \cite{ghy93} extended this result, showing that \( f \) is \( C^\infty \)-rigid if the distributions are \( C^{1+\mathrm{Lip}} \)-regular (see also \cite{kh90,fk91} under the volume-preserving assumption). In higher dimensions, \cite{bl93} proved that any Anosov diffeomorphism with \( C^\infty \) distributions, which preserves an affine connection, is \( C^\infty \)-conjugate to an infra-nilmanifold automorphism.

Our main result follows a similar theme, but with much lower regularity, using Hölder continuity as the threshold for regularity:

\begin{theorem}\label{theorem 3dim rigid}
Let \( f \in \Diff^\infty_{\mathrm{vol}}(\mathbb{T}^3) \) be a partially hyperbolic Anosov diffeomorphism. Then \( f \) is \( C^\infty \)-rigid if and only if the stable distribution \( E^s \) and the center distribution \( E^c \) are \( \theta_s+ \)- and \( \theta_c+ \)-Hölder continuous, respectively.
\end{theorem}

Recall that \( E^s \) and \( E^c \) for a partially hyperbolic Anosov diffeomorphism \( f \in \Diff^\infty_{\mathrm{vol}}(\mathbb{T}^3) \) are a priori \( \theta_s-\)- and \( \theta_c-\)-Hölder continuous, respectively. Therefore, our rigidity result can also be viewed as an (almost) sharp bootstrap result for the regularity of invariant distributions. 

Deducing rigidity result from Hölder regularity is rather rare. For example, for any $\epsilon>0$ and any partially hyperbolic Anosov automorphism $L$ on $\TT^3$, we can find a volume preserving diffeomorphism $f$ which is $C^1$-close to $L$, such that the conjugacy $h$ given by $h \circ f = h\circ L$ is bi-$(1-\epsilon)$-H\"older but not $C^1$. For more discussion about diffeomorphisms H\"older conjugate to Anosov diffeomorphisms, see \cite{F06,gog2010}. 

\begin{remark}\label{remark rigidity}
The regularity assumptions on \( E^s \) and \( E^c \), the volume-preserving condition, and the three-dimensional condition are all necessary, as the following examples illustrate:

\begin{itemize}
    \item There exists a partially hyperbolic Anosov diffeomorphism \( f \in \Diff^\infty_{\mathrm{vol}}(\mathbb{T}^3) \) such that \( E^s \) is \( C^1 \) but \( f \) is not \( C^1 \)-conjugate to a linear automorphism \( L \) (see Remark 5.2 in \cite{gs20}). Thus, high Hölder regularity of \( E^s \) alone is not sufficient for smooth rigidity.
    
    \item There exists a partially hyperbolic Anosov diffeomorphism \( f \in \Diff^\infty_{\mathrm{vol}}(\mathbb{T}^3) \) such that \( E^c \) is \( C^1 \) but \( f \) is not \( C^1 \)-conjugate to \( L \) (see Section 4.2 in \cite{var16}). Therefore, high Hölder regularity of \( E^c \) alone is also not sufficient for smooth rigidity.
    
    \item There exists an Anosov diffeomorphism \( f: \mathbb{T}^3 \to \mathbb{T}^3 \) such that \( E^u, E^c, E^s \) are all \( C^1 \) but \( f \) is not \( C^1 \)-conjugate to \( L \) (see Example \ref{example c1}). Therefore, the volume-preserving assumption is also necessary for smooth rigidity.
    
    \item For any \( n \neq 3 \), there exists a volume-preserving partially hyperbolic Anosov diffeomorphism \( f \in \Diff^\infty(\mathbb{T}^n) \) with all invariant distributions \( C^1 \), but \( f \) is not \( C^1 \)-conjugate to a linear automorphism \( L \) (see Example \ref{example dim}). 
\end{itemize}
\end{remark}

\subsection{\textit{Stable $C^1$ or stable fractal} conjectures}

From our results and proofs, it appears that the fractal nature of distributions in hyperbolic and partially hyperbolic systems arises from certain transversality properties of dynamically-defined objects, which are expected to be stable under \( C^1 \)-small perturbations. On the other hand, the bunching regime naturally exhibits \( C^1 \)-stability.

These findings, together with the fragility of rigidity phenomena, motivate a series of conjectures concerning the invariant distributions of Anosov and partially hyperbolic systems. For instance, let \( f \) be a \( C^r \) partially hyperbolic diffeomorphism on a compact manifold \( M \). We say that the distribution \( E^s = E^s_f \) is \emph{stably \( C^1 \)} if, for any \( \tilde{f} \in \Diff^r(M) \) that is \( C^1 \)-close to \( f \), the distribution \( E^s_{\tilde{f}} \) remains \( C^1 \). We say \( E^s \) is \emph{stably fractal} if, for any \( \tilde{f} \) that is \( C^1 \)-close to \( f \), the distribution \( E^s_{\tilde{f}} \) exhibits a fractal graph in the sense of Definition \ref{def: frct grph}. Similar stable properties can be defined for \( E^u \) and \( E^c \). 

We then propose the following conjecture for partially hyperbolic systems:

\begin{conjecture}[Stably \( C^1 \) or Stably Fractal Conjecture for Partially Hyperbolic Systems]\label{conj: main PH}
Among \( C^r \), \( r \geq 2 \), partially hyperbolic diffeomorphisms, there exists a \( C^1 \)-open, \( C^r \)-dense set \( \Omega \) such that for each \( f \in \Omega \), the distribution \( E^\ast_f \) for \( \ast \in \{s, u, c\} \) is either stably \( C^1 \) or stably fractal. Moreover, if \( E^u_f \) and \( E^s_f \) are both non-trivial, then \( E^c_f \) is stably fractal.
\end{conjecture}

The ``moreover" part of the conjecture reflects our belief that the \( C^1 \) center distribution is fragile. Similarly, we have the following corresponding conjecture for Anosov systems:

\begin{conjecture}[Stably \( C^1 \) or Stably Fractal Conjecture for Anosov Systems]\label{conj: main Anosov}
Among \( C^r \), \( r \geq 2 \), Anosov diffeomorphisms, there exists a \( C^1 \)-open, \( C^r \)-dense set such that the (un)stable distribution of each element is either stably \( C^1 \) or stably fractal.
\end{conjecture}

We can also pose the same conjectures for volume-preserving (symplectic) partially hyperbolic and Anosov diffeomorphisms.

\subsection{More historical remarks and problems}

Regarding the dimension theory of graphs of functions, in addition to the works mentioned earlier, \cite{diaz19} recently studied the box dimension of function graphs over Anosov systems on surfaces, one-dimensional attractors in the base, and hyperbolic Cantor sets. It provided a formula for the box dimension in terms of appropriate pressure functions. Other related studies include \cite{bed89, kap84, hnw02, hl93}.

A natural question arising from our results is whether the fractal property of invariant distributions can be characterized in terms of Hausdorff dimension, instead of box dimension. This would lead to the problem of determining the Hausdorff dimension of dynamically invariant fractal sets in non-conformal and non-algebraic settings—a challenging problem in general, which can be traced back to the foundational work of Falconer \cite{fal88}. For the Hausdorff dimension of lower-dimensional invariant sets defined by non-conformal dynamics, see \cite{BHR, HR, R} for results in the affine setting and \cite{LPX, JLPX} for the projective case, following the breakthrough \cite{hoc14}. A classical result related to fractal versus rigidity phenomena in dimension theory is Bowen's celebrated theorem on the dimension rigidity of quasicircles \cite{bow79}. For a more recent higher-dimensional non-conformal generalization of Bowen's result, see \cite{LPX, JLPX}.

It would also be intriguing to explore whether our results can be generalized to non-uniform settings, higher-dimensional spaces, or skew-product settings. For example, one might ask whether the center distribution in the famous pathological example of Shub and Wilkinson \cite{sw00} has a fractal graph.

\subsection{Organization of the paper}  
The structure of this paper is as follows: in Sections \ref{sec: non-frac easy} and \ref{sec: non-frac tech vers} we introduce various formulations and applications of the non-fractal invariance principle.  Section \ref{sec: prel} provides the necessary preliminaries for the proofs. In Section \ref{sec graph} we prove Theorem \ref{theorem cocycle pointwise}. In the study of fractal geometry \cite{diaz19,hnw02} or hyperbolic systems \cite{has95,hw99}, it is common to use obstruction functions to obtain upper bounds on the regularity of maps. Here, we extend this approach by introducing a more flexible framework that incorporates local information. This generalization allows us to handle broader classes of systems, including partially hyperbolic systems.  Section \ref{sec: proof remain}  concludes with the proofs of all remaining results.  

\subsection*{Acknowledgements.} 
We thank Artur Avila for his useful comments and encouragement during the early stages of this work. We also thank Amie Wilkinson for useful discussions and for explaining the ideas in \cite{hw99}. We thank Haojie Ren for useful discussions and for pointing out the reference \cite{hl93}. We are grateful to Lorenzo J. D\'iaz and Katrin Gelfert for their encouragement to write this paper. We also thank Jonathan Dewitt, Bassam Fayad, Ziqiang Feng, Shaobo Gan, Ruihao Gu, Yuxiang Jiao, Yi Shi, Yun Yang and Zhiyuan Zhang for helpful discussions. D. X. is supported by National Key R$\&$D Program of China No. 2024YFA1015100, NSFC 12090010 and 12090015.
\section{Notations and the statement of non-fractal invariance principle}\label{sec: non-frac easy}
The goal of this section is to state our non-fractal invariance principle which was mentioned in Section \ref{subsec: gene dich}. For a more technical version of it see Section \ref{sec: non-frac tech ver}.

We first fix some notations. For a linear map $A$ between normed linear spaces we denote by $m(A):=\min_{\|v\|=1}\|A\cdot v\|$. A topological foliation $W$ of a compact manifold $M$ is called minimal if every $W$-leaf is dense in $M$. 

Let $\pi:N\to M$ be a fiber bundle over a compact manifold $M$, a section is defined by a subset of $N$ such that it intersects every fiber at a single point, hence could be (uniquely) identified to a map $\Phi: M\to N$ such that $\pi\circ \Phi=id$. If $\Phi$ is continuous ($C^r$ resp. ) we call $\Phi$ is a continuous ($C^r$ resp. ) section. 

Let \(f \in \Diff^r(M)\) (\(r \geq 1\)) be a partially hyperbolic diffeomorphism with a dominated splitting \(TM = E^s \oplus E^c \oplus E^u\), where \(E^c\) and \(E^s\) could be trivial. Let \(N\) be a $C^r$ fiber bundle over \(M\) with compact fibers. A bundle morphism \(F: N \to N\) over \(f: M \to M\) is called \textit{partially hyperbolic} if \(F\) expands \(N\) along the fibers weaker than \(f\) does along \(W^u_f\). Specifically, there exists \(k \in \mathbb{Z}^+\) such that
\[
1<\|F^k_x\| := \sup_{y \in N_x}\|D F^k_x(y)\| < m(D f^k|_{E^u(x)}), \quad \forall x \in M.
\]
By the classical stable manifold theorem for partially hyperbolic systems (see, for example, \cite{hps77}), for every point in \(N\), there exists a globally defined unstable leaf \(W^u_F\) passing through it, which projects to a \(W^u_f\)-leaf on the base. Moreover, for any two \(N\)-fibers \(N_x\) and \(N_y\) where \(x\) and \(y\) belong to the same local \(W^u_f\)-leaf, the unstable holonomy map of \(W^u_F\),
\[
(h^u_F)_{xy}: N_x \to N_y, \quad (h^u_F)_{xy}(z) := W^u_{F, \mathrm{loc}}(z) \cap N_y,
\]
is well-defined. Then the non-fractal invariance principle could be stated as follows:
\begin{theorem}[Non-fractal invariance principle]\label{them: non fract easy vers}Let $f\in \Diff^r(M),r\geq1$ be a partially hyperbolic diffeomorphism on a compact manifold and $W^u_f$ is minimal. Let $\pi: N\to M$ be a $C^r$ fiber bundle over $M$ and $F:N\to N$ be a partially hyperbolic bundle morphism projects to $f$. Then for any continuous $F$-invariant section $\Phi:M\to N$, if $\Phi$ has no fractal graph, then \(\Phi\) is uniformly \(C^r\) along \(W^u_f\) and invariant under \(h^u_F\).
\end{theorem}


\section{Variants of non-fractal invariance principle} \label{sec: non-frac tech vers}

\subsection{Quantitative version of non-fractal invariance principle }\label{sec: non-frac tech ver}
Let $F:N\to N, f:M\to M$ be as in Theorem \ref{them: non fract easy vers}. To state the result more precisely, for \(k \in \mathbb{Z}^+\) we define
\[
A_k = \sup_{x \in M} \frac{\log \|F^k_x\|}{\log m(D f^k|_{E^u(x)})}, \quad A := \inf_{k \in \mathbb{Z}^+} A_k,
\]
and
\[
\alpha(k, x) := \liminf_{n \to +\infty} \frac{\sum_{i=1}^n \log \|F^k_{f^{-ik}x}\|}{\sum_{i=1}^n \log m(D f^k|_{E^u(f^{-ik}x)})}, \quad \alpha := \inf_{x \in M, k \in \mathbb{Z}^+} \alpha(k, x).
\]
It follows from the definitions that \(\alpha(k, x) \leq A_k\) and \(0<\alpha \leq A < 1\).

\begin{theorem}\label{theorem cocycle pointwise}
Suppose that \(F: N \to N\) is partially hyperbolic and \(W^u_f\) is minimal. Then for any continuous \(F\)-invariant section \(\Phi: M \to N\),
\begin{enumerate}
    \item If \(\Phi\) is uniformly \(\alpha+\)-H\"{o}lder continuous along \(W^u_f\), then \(\Phi\) is invariant under \(h^u_F\) and is uniformly \(C^r\) along \(W^u_f\).
    \item If (1) does not hold, then \(\underline{\dim}_B \mathrm{Graph}(\Phi) \geq \dim M + 1 - A > \dim M\).
\end{enumerate}
\end{theorem}


\subsection{Applications: distributions of partially hyperbolic Anosov systems}\label{sec boot high dim}
Let $f$ be a $C^{2}$ partially hyperbolic Anosov diffeomorphism on a compact manifold $M$. Define  
\[
\alpha_s(k, x) := \liminf_{n \to +\infty}  
\frac{\sum_{i=1}^n \left( \log\|Df^k|_{E^{c}(f^{-ik}x)}\| - \log m(Df^k|_{E^{s}(f^{-ik}x)}) \right)}{\sum_{i=1}^n \log m(Df^k|_{E^u(f^{-ik}x)})},
\]
and  
\[
\alpha_s := \inf_{x \in M, k \in \mathbb{Z}^+} \alpha_s(k, x).
\]
For the stable bundle $E^s$, we have:
\begin{theorem}\label{theorem s bunching}
Suppose $f$ is topologically transitive. Assuming there exists a constant $k \in \mathbb{Z}^+$ such that  
\[
\frac{\|Df^k|_{E^{c}(x)}\|}{m(Df^k|_{E^{s}(x)})} < m(Df^k|_{E^u(x)}) \quad \text{and} \quad 
\frac{\|Df^k|_{E^{u}(x)}\|}{m(Df^k|_{E^{u}(x)})} < \|Df^{k}|_{E^c(x)}\|^{-1},
\]
then $\alpha_s \in (0,1)$, and the following holds:  
\begin{enumerate}
    \item If $E^s$ is $\alpha_s+$-H\"older continuous, then $E^s$ is $C^{1}$, and $E^s \oplus E^u$ is jointly integrable.  
    \item If (1) does not hold, then $\underline{\dim}_{B} \mathrm{Graph}(E^s) > \dim M$.  
\end{enumerate}
\end{theorem}
The assumption is not too restrictive. For instance, $f$ satisfies all the above assumptions if it is $C^1$-close to a volume-preserving partially hyperbolic Anosov diffeomorphism with $\dim E^u = 1$.  

Similarly, define  
\[
\alpha_c(k, x) := \liminf_{n \to +\infty}  
\frac{\sum_{i=1}^n \left( \log\|Df^{-k}|_{E^{s}(f^{ik}x)}\| - \log m(Df^{-k}|_{E^{c}(f^{ik}x)}) \right)}{\sum_{i=1}^n \log m(Df^{-k}|_{E^s(f^{ik}x)})}, 
\]
and  
\[
\alpha_c := \inf_{x \in M, k \in \mathbb{Z}^+} \alpha_c(k, x).
\]
For the central bundle $E^c$, we have:
\begin{theorem}\label{theorem c bunching ss}
Suppose $W^s$ is minimal and there exists a constant $k \in \mathbb{N}$ such that  
\[
\frac{\|Df^k|_{E^{c}(x)}\|}{m(Df^k|_{E^{s}(x)})} < \min\{m(Df^k|_{E^u(x)}),\|Df^{k}|_{E^s(x)}\|^{-1}\},
\]
then $\alpha_c \in (0,1)$, and the following holds:  
\begin{enumerate}
    \item If $E^c$ is uniformly $\alpha_c+$-H\"older continuous along $W^u$, then $E^c$ is uniformly $C^{1}$ along $W^u$.  
    \item If (1) does not hold, then $\underline{\dim}_{B} \mathrm{Graph}(E^c) > \dim M$.  
\end{enumerate}
\end{theorem}
The assumption is also not too restrictive. For instance, $f$ satisfies all the above assumptions if $\dim E^s = \dim E^u= 1$.  
\section{Preliminaries}\label{sec: prel}
\subsection{Dimension of graphs}\label{sec kdim}
We recall some facts about the box dimension (see \cite{mat95} for further details). Let $A$ be a non-empty bounded subset of $\RR^n$. For any $\epsilon > 0$, let $N(A,\epsilon)$ denote the smallest number of $\epsilon$-balls required to cover $A$. The \textit{lower box dimension} and \textit{upper box dimension} of $A$ are defined by:
\[
\underline{\dim}_B A := \liminf_{\epsilon \to 0} \frac{\log N(A,\epsilon)}{-\log \epsilon}, \quad
\overline{\dim}_B A := \limsup_{\epsilon \to 0} \frac{\log N(A,\epsilon)}{-\log \epsilon}.
\]
Clearly, $\underline{\dim}_B A \leq \overline{\dim}_B A$. If the two values coincide, their common value is called the \textit{box dimension} of $A$. Alternatively, the number \( N(A,\epsilon) \) in the definition above can be replaced by the packing number \( P(A,\epsilon) \), which is the largest number of disjoint $\epsilon$-balls with centers in \( A \). 

Next, we focus on the dimension of the graph of a continuous section over a smooth manifold. Let \(\pi: N \to M\) be a $C^1$ vector bundle, and let \(\Phi: M \to N\) be a continuous section. The H\"older continuity of \(\Phi\) provides an upper bound on the box dimension of its graph:

\begin{proposition}
If \(\Phi\) is \(\beta\)-H\"older, then 
\[
\overline{\dim}_B \mathrm{Graph}(\Phi) \leq \dim M + (1-\beta)(\dim N - \dim M).
\]
\end{proposition}

\begin{proof}
For any \(\epsilon > 0\), let \( A_{\epsilon} \subset M \) be a set such that \(\cup_{x \in A_\epsilon} B(x,\epsilon) = M\), where \( B(x,\epsilon) \) denotes the open ball of radius \(\epsilon\) centered at \(x\). We choose \( A_\epsilon \) such that \(\#A_\epsilon < C_0 \epsilon^{-\dim M}\) for some constant \( C_0 \).

Since \(\Phi\) is \(\beta\)-H\"older, for each \(x \in A_\epsilon\), 
$$N(x,\epsilon):=N(\mathrm{Graph}(\Phi|_{B(x,\epsilon)}), \epsilon) \leq C_1 \epsilon^{(\beta-1)(\dim N - \dim M)}$$
for some uniform constant $C_1$. Consequently, 
\[
\overline{\dim}_B \mathrm{Graph}(\Phi) \leq \limsup_{\epsilon \to 0} \frac{\log \sum_{x \in A_\epsilon} N(x,\epsilon) }{-\log \epsilon}  \leq \dim M + (1-\beta)(\dim N - \dim M).
\]\end{proof}
\subsection{Regularity of distributions and foliations}
Let $M$ be a manifold of dimension $n \geq 2$, and let $F$ be a foliation of $M$. For any $r \geq 0$, there are three common ways to measure the regularity of $F$:  
\begin{enumerate}[(a)]
    \item $F$ is tangent to a $C^r$ distribution.  
    \item The foliation charts are $C^r$ diffeomorphisms.  
    \item The leaves and local holonomy maps are uniformly $C^r$.  
\end{enumerate}
Typically, we say that $F$ is a $C^r$ foliation if (b) holds. For $r \geq 1$, it is known that $\text{(a)} \Rightarrow \text{(b)} \Rightarrow \text{(c)}$, but in general, the converse implications do not hold \cite{psw97}.  For $r \geq 1$ and $r \notin \{2, 3, \dots\}$, we have $\text{(c)} \Rightarrow \text{(b)}$, which follows as a corollary of a regularity lemma by Journ\'e \cite{jou88}:  
\begin{proposition}[Journ\'{e}'s Lemma]\label{prop journee}
Let $F$ and $G$ be two transverse foliations of $M$ with uniformly $C^r$ leaves for some $r \geq 1$ and $r \notin \{2, 3, \dots\}$. If a continuous function $f: M \to \mathbb{R}$ is $C^r$ along the $F$-leaves and $G$-leaves, then $f$ is $C^r$.  
\end{proposition}

\subsection{Invariant distributions of partially hyperbolic Anosov diffeomorphisms}\label{sec pre anosov}  
In this subsection, we summarize some results concerning the invariant distributions and foliations of partially hyperbolic Anosov diffeomorphisms.  \begin{proposition}\label{prop ano}
Let $f \in\Diff^r(M),r\geq 2$ be a partially hyperbolic Anosov diffeomorphism that contracts $E^c$. Then:  
\begin{itemize}
    \item $E^s$, $E^u$, and $E^{cs}$ are uniquely integrable to continuous foliations $W^s$, $W^u$, and $W^{cs}$ respectively, with uniformly $C^r$ leaves \cite{hps77}.  
    \item $W^s$ $C^r$-subfoliates $W^{cs}$ \cite{hps77}.  
    \item If $E^{cs}$ and $E^u$ are bunching, i.e., there exists $k \in \mathbb{Z}^+$ such that  
    \[
    \frac{\|Df^k|_{E^u(x)}\|}{m(Df^k|_{E^u(x)})} < \|Df^k|_{E^{c}(x)}\|^{-1}, \quad  
    \frac{\|Df^k|_{E^{c}(x)}\|}{m(Df^k|_{E^{s}(x)})} < m(Df^k|_{E^u(x)}), \quad \forall x \in M,
    \]
then $E^{cs}$ and $E^u$ are $C^{1}$ respectively \cite{has95}, the center-stable holonomy $h^{cs}$ and the unstable holonomy $h^u$ are uniformly $C^{1}$ respectively \cite{psw97}.  
\end{itemize}
\end{proposition}
Moreover, for $M = \mathbb{T}^3$, we have:  
\begin{proposition}\label{prop anosov t3}
Let $f \in\Diff^r(\mathbb{T}^3),r\geq 2$ be a partially hyperbolic Anosov diffeomorphism that contracts $E^c$. Then:  
\begin{itemize}
    \item $f$ is dynamically coherent, i.e., $E^{cs}$, $E^{cu} $, and $E^c$ are integrable to foliations whose leaves are $C^{1}$ \cite{bbi09,por15} (but may fail to be $C^2$ even when $r = \infty$ \cite{jpl95}).  
    \item If $E^c$ is $r$-dominated, i.e., there exist constants $c > 0$ and $\lambda >1$ such that  
    \[
    |Df^n|_{E^c(x)}|^r \geq c \lambda^n |Df^n|_{E^s(x)}|,\quad \forall x \in M, n \in \ZZ^+,
    \]
then the leaves of $W^c$ are $C^r$ immersed manifolds \cite{hps77}.  
    \item $E^{cs}$ is always bunching and $C^{1}$ \cite{has95}.  
    \item $W^u$ and $W^{cs}$ are minimal \cite{fra69, man74}.  
    \item If $E^s \oplus E^u$ is jointly integrable, then $W^s$ is minimal \cite{rgz17}.  
\end{itemize}
\end{proposition}
\subsection{Periodic data}\label{sec periodic}
It is known that a $C^\infty$ Anosov diffeomorphism $f$ on $\mathbb{T}^m$ is topologically conjugate to a linear Anosov automorphism $L$ \cite{fra69, man74}. However, the conjugacy may fail to be $C^1$. An obstruction arises because, if $f$ is $C^1$ conjugate to $L$, then $f$ and $L$ must have the same \textit{periodic data}, i.e., $Df^n(p)$ is conjugate to $DL^n(p)$ for every periodic point $p$ with $f^n(p) = p$.  

For $m = 2, 3$, periodic data is the only obstruction, i.e., $f$ is $C^\infty$ conjugate to $L$ if and only if they have the same periodic data \cite{lmmt2,lla92, dg24}. However, for $m \geq 4$, $f$ may not be $C^1$ conjugate to $L$ even if they share the same periodic data \cite{lla92}.  

Let $f: \mathbb{T}^3 \to \mathbb{T}^3$ be a partially hyperbolic Anosov diffeomorphism. We say that $f$ and $L$ have the same $\ast$-data ($\ast = s, c, u$) if $Df^n|_{E^\ast}(p) = DL^n|_{E^\ast}(p)$ for every periodic point $p$ with $f^n(p) = p$. We summarize some results about $\ast$-data below:  
\begin{proposition}\label{prop periodic data}
Let $f \in\Diff^2(\mathbb{T}^3)$ be a partially hyperbolic Anosov diffeomorphism that contracts $E^c$. Then:  
\begin{itemize}
    \item If $f$ is volume preserving, then $W^c$ is Lipschitz continuous if and only if $f$ has the same $s$-periodic data as $L$ \cite{gog12}.
    \item $E^s \oplus E^u$ is jointly integrable if and only if $f$ has the same $c$-periodic data as $L$ \cite{gs20}.
    \item For $f \in\Diff^{2+}(\mathbb{T}^3)$, $E^{cs}$ is a $C^{2+}$ distribution if and only if $f$ has the same $u$-periodic data as $L$ \cite{gu23}.
    \item $f \in\Diff^2(\mathbb{T}^3)$ is $C^\infty$ conjugate to $L$ if and only if $f$ has the same periodic data as $L$ \cite{gg08, gog17}.
\end{itemize}
\end{proposition}

\section{Invariant sections over partially hyperbolic systems}\label{sec graph}
The aim of this section is to prove Theorem \ref{theorem cocycle pointwise}. 
\subsection{Excessive H\"older regularity implies holonomy invariance}
Recall that $(h^u_F)_{xy}: N_x \to N_y$ is the unstable holonomy map of $W^u_F$. For every $x \in M$, define  
$$\gamma_x: W^u_\text{loc}(x) \to N, \quad y \mapsto (h^u_F)_{xy} (\Phi(x)),$$  
where $\gamma_x$ is uniformly $C^r$.  

To describe the holonomy invariance of $\Phi$, it is convenient to define the following \textit{obstruction function} (see also \cite{hnw02, diaz19}): given $x \in M$ and $\delta > 0$, let  
$$\Delta_\delta(x):= \sup_{t \in W^{u}_\delta(x)} d(\Phi(t), \gamma_x(t)),$$
where $W^{u}_\delta(x)$ denotes the unstable disk centered at $x$ with radius $\delta$.
\begin{lemma}\label{lemma bounded type}
For any $\epsilon > 0$ and $k \in \mathbb{Z}^+$, there exists a constant $\delta > 0$ such that for any $x \in M$, $t \in W^u_\delta(x)$, and $n \in \mathbb{Z}^+$, the following inequalities hold:
$$
\|F^{nk}_{f^{-nk}t}\| \leq \prod_{i=1}^{n} \|F^k_{f^{-ik}x}\|^{1+\epsilon}, \quad 
d(f^{-nk}x, f^{-nk}t) \leq \delta \prod_{i=1}^{n} m(D f^k|_{E^u(f^{-ik}x)})^{-1+\epsilon}.
$$
\end{lemma}
\begin{proof}
By the continuity of $F$ and $D f$, for any $\epsilon > 0$ and $k \in \mathbb{Z}^+$, there exists a constant $\delta > 0$ such that for every $x \in M$ and $t \in W^u_\delta(x)$,
$$\|F_t^k\| < \|F_x^k\|^{1+\epsilon}, \quad \| D f^{-k}|_{E^u(t)}\| < \| D f^{-k}|_{E^u(x)}\|^{1-\epsilon},$$
which implies 
$$\|F^{nk}_{f^{-nk}t}\| \leq \prod_{i=0}^{n-1} \|F^k_{f^{(-n+i)k}t}\| \leq \prod_{i=0}^{n-1} \|F^k_{f^{(-n+i)k}x}\|^{1+\epsilon} = \prod_{i=1}^{n}\|F^k_{f^{-ik}x}\|^{1+\epsilon}$$
and
\begin{eqnarray*}
  d(f^{-nk}x, f^{-nk}t) & \leq & d(x,t) \sup_{\tau \in W^u_\delta(x)} \| D f^{-nk}|_{E^u(\tau)}\| \leq  \delta \prod_{i=0}^{n-1} \sup_{\tau \in W^u_\delta(f^{-ik}x)} \| D f^{-k}|_{E^u(f^{-ik}\tau)}\| \\
  & < & \delta \prod_{i=0}^{n-1} \| D f^{-k}|_{E^u(f^{-ik}x)}\|^{1-\epsilon} =  \delta \prod_{i=1}^{n} m( D f^k|_{E^u(f^{-ik}x)})^{-1+\epsilon}
\end{eqnarray*}
for every $n \in \mathbb{Z}^+$. 
\end{proof}
\begin{lemma}\label{lemma exce holder}
If $\Phi$ is $\theta$-H\"older with $\theta > \alpha(k,x)$ for some $x \in M$ and $k \in \mathbb{Z}^+$, then there exists some $\delta > 0$ such that $\Delta_\delta(x) = 0$.
\end{lemma}
\begin{proof}
For any $\epsilon > 0$ with $\alpha(k,x) < \frac{1-\epsilon}{1+\epsilon} \theta$, choose $\delta$ small enough such that Lemma \ref{lemma bounded type} holds. Denote by $C_0$ and $L_0$ the H\"older constant and Lipschitz constant of $\Phi$ and $\gamma_x$ respectively. For any $t \in W^u_\delta(x)$ and $n \in \mathbb{Z}^+$, we have
\begin{eqnarray*}
    d(\Phi(t), \gamma_x(t)) & = & d\left( F^{nk}(\Phi(f^{-nk}t)), F^{nk}(\gamma_{f^{-nk}x}(f^{-nk}t)) \right) \\
    & \leq & \|F^{nk}_{f^{-nk}t}\| \cdot d\left( \Phi(f^{-nk}t), \gamma_{f^{-nk}x}(f^{-nk}t) \right).
\end{eqnarray*}
Note that $\gamma_x(x) = \Phi(x)$ for every $x \in M$, so the second term is estimated by
\begin{eqnarray*}
    d\left( \Phi(f^{-nk}t), \gamma_{f^{-nk}x}(f^{-nk}t) \right) & \leq & d\left( \Phi(f^{-nk}t), \Phi(f^{-nk}x) \right) + d\left( \gamma_{f^{-nk}x}(f^{-nk}x), \gamma_{f^{-nk}x}(f^{-nk}t) \right) \\
    & \leq & C_0 \cdot d(f^{-nk}x, f^{-nk}t)^{\theta} + L_0 \cdot d(f^{-nk}x, f^{-nk}t) \\
    & \leq & C \cdot d(f^{-nk}x, f^{-nk}t)^{\theta}
\end{eqnarray*}
for some constant $C > 0$ and sufficiently large $n$. By Lemma \ref{lemma bounded type} and the definition of $\alpha(k,x)$, we obtain
\begin{eqnarray*}
    d(\Phi(t), \gamma_x(t)) & \leq & C \liminf_{n \to \infty} \prod_{i=1}^{n} \|F^k_{f^{-ik}x}\|^{1+\epsilon} \cdot d\left( \Phi(f^{-nk}t), \gamma_{f^{-nk}x}(f^{-nk}t) \right) \\
    & \leq & C \liminf_{n \to \infty} \prod_{i=1}^{n} m\left( D f^k|_{E^u(f^{-ik}x)} \right)^{\alpha(k,x)(1+\epsilon) + \theta(\epsilon - 1)} = 0.
\end{eqnarray*}
Therefore, $\Delta_\delta(x) = 0$. 
\end{proof}
If $W^u$ is minimal, then $\Delta_\delta(x) = 0$ implies that $\Delta_\delta$ vanishes on $M$. 
\begin{lemma}\label{lemma delta}
If $W^u$ is minimal, then for any $\delta > 0$, either $\Delta_\delta(x) = 0$ for every $x \in M$, or $\Delta_\delta(x) > C_0 = C_0(\delta)$ is uniformly bounded away from $0$.  
\end{lemma}
\begin{proof}
For any fixed $\delta$, the obstruction function $\Delta_\delta : M \to \mathbb{R}$ is continuous. If $\Delta_\delta(x)$ is not uniformly bounded away from $0$, by taking a limit, there exists some $x_0 \in M$ such that $\Delta_\delta(x_0) = 0$. In particular, $\Delta_{\delta/3}(t) = 0$ for any $t \in W^{u}_{\delta/6}(x_0)$. By taking an iteration, $\Delta_{\delta/3} = 0$ on the dense subset $\bigcup_{n \in \mathbb{N}} f^n(W^{u}_{\delta/6}(x_0))$, and hence $\Delta_{\delta/3} = 0$ on the entire manifold $M$. This implies that $\Phi|_{W^{u}_{\text{loc}}(x)} = \gamma_x$ and $\Delta_\delta(x) = 0$ for every $x \in M$. 
\end{proof}
\begin{proposition}
   If $\Phi$ is $\alpha+$-H\"older continuous and $W^u$ is minimal, then $\Phi$ is invariant under $h^u_F$ and is uniformly $C^r$ along $W^u_f$.
\end{proposition}

\begin{proof}
   By Lemma \ref{lemma exce holder}, if $\Phi$ is $\alpha+$-H\"older continuous, then $\Delta_\delta(x) = 0$ for some $\delta > 0$ and $x \in M$. By Lemma \ref{lemma delta}, $\Delta_\delta = 0$ holds everywhere. Therefore, $\Phi$ is invariant under $h^u_F$ and is uniformly $C^r$ along $W^u_f$.
\end{proof}

\subsection{Non-invariance implies fractal graph}
Before going into the proof, we introduce a general criterion for the graph $\Phi$ to be fractal. For any \(\epsilon > 0\), \(\alpha \in (0,1)\), and \(x \in M\), define:
\[
\text{osc}(\Phi, x, \epsilon) := \sup_{y \in M, \, d(x, y) < \epsilon} d(\Phi(x), \Phi(y)), \quad h_{\alpha, \epsilon}(x) := \frac{\text{osc}(\Phi, x, \epsilon)}{\epsilon^\alpha}.
\]
\begin{lemma}\label{lemma dim criterion}
If there are constants \(\epsilon_0 > 0\) and \(c > 0\) such that \(h_{\alpha, \epsilon}(x) > c\) for some \(\alpha \in (0,1)\), every \(0 < \epsilon < \epsilon_0\), and every \(x \in M\), then 
\[
\underline{\dim}_B \mathrm{Graph}(\Phi) \geq \dim M + 1 - \alpha.
\]\end{lemma}
\begin{proof}
    For any \(0 < \epsilon < \epsilon_0\), let \(A_{2\epsilon} \subset M\) be a \(2\epsilon\)-packing subset, i.e., \(B(x, 2\epsilon) \cap B(y, 2\epsilon) = \emptyset\) for any \(x \neq y \in A_{2\epsilon}\). We may choose \(A_{2\epsilon}\) such that 
$\#A_{2\epsilon} \geq c_0 \epsilon^{-\dim M}$ for some constant \(c_0 > 0\). 

For any \(x \in A_{2\epsilon}\), since \(h_{\alpha, \epsilon}(x) > c\), we can find an \(\epsilon\)-packing subset \(H_\epsilon(x) \subset \Phi|_{B(x, \epsilon)}\) such that \(\#H_\epsilon(x) \geq c \epsilon^{\alpha - 1}/2\). Note that \(d(H_\epsilon(x), H_\epsilon(y)) > \epsilon\) for any \(x \neq y \in A_{2\epsilon}\). Thus, we have 
\[\underline \dim_B \mathrm{Graph}(\Phi) \geq \liminf_{\epsilon \to 0} \frac{\log \sum_{x \in A_{2\epsilon}} H_\epsilon(x)}{-\log \epsilon} \geq \dim M + 1 - \alpha.\]
\end{proof}
We now show that the graph of $\Phi$ is fractal if $\Phi$ is not invariant under $h^u_F$.
\begin{proposition}\label{prop low holder frac}
   If $\Phi$ is not $\alpha+$-H\"older continuous, then $\underline{\dim}_B \mathrm{Graph}(\Phi) \geq \dim M + 1 - A$.
\end{proposition}
\begin{proof}
It suffices to prove that $\underline{\dim}_B \mathrm{Graph}(\Phi) \geq \dim M + 1 - A_k$ for every $k \in \mathbb{Z}^+$. By taking a finite iteration, it suffices to prove this for $A_1 < 1$. Choose $\epsilon > 0$ small enough such that $\theta := \frac{1 + \epsilon}{1 - \epsilon} A_1 < 1$, and choose $\delta > 0$ such that Lemma \ref{lemma bounded type} holds for $\epsilon$ and $k = 1$.

Since $\Phi$ is not invariant under $h^u_F$, by Lemma \ref{lemma delta}, we know that $\Delta_\delta \geq C_0 = C_0(\delta)$ is uniformly bounded away from $0$. For any $x \in M$ and $n \in \mathbb{Z}^+$, we can pick a point $t_n \in W^u_\text{loc}(x)$ such that $f^n t_n \in W^u_\delta(f^n x)$ and $d(\Phi(f^n t_n), \gamma_{f^n x}(f^n t_n)) \geq C_0 / 2$. By Lemma \ref{lemma bounded type} and the definition of $A_1$, we have
\[
d(\Phi(t_n), \gamma_x(t_n)) \geq \frac{C_0}{2} \|F^n_{t_n}\|^{-1} \geq \frac{C_0}{2} \prod_{i=0}^{n-1} \|F_{f^i x}\|^{-(1 + \epsilon)} \geq \frac{C_0}{2} \prod_{i=0}^{n-1} m(Df|_{E^u(f^i x)})^{-A_1(1 + \epsilon)}.
\]
Denote $r_n(x) := \delta \prod_{i=0}^{n-1} m(Df|_{E^u(f^i x)})^{ - 1+\epsilon}$, then $r_n(x) / r_{n+1}(x) < C_1$ is uniformly bounded for $x \in M, n \in \ZZ^+$ and some constant $C_1$. By Lemma \ref{lemma bounded type}, we also have $d(x, t_n) \leq r_n(x)$.

Since $\gamma_x$ is Lipschitz for some constant $L$, by triangle inequality: 
\[
d(\Phi(t_n), \Phi(x)) \geq d(\Phi(t_n), \gamma_x(t_n)) - d(\Phi(x), \gamma_x(x)) - d(\gamma_x(x), \gamma_x(t_n)) 
\]
\[
\geq \frac{C_0}{2} \prod_{i=0}^{n-1} m(Df|_{E^u(f^i x)})^{-A_1(1 + \epsilon)} - 0 - L r_n(x).
\]
This implies:
\[
\frac{d(\Phi(t_n), \Phi(x))}{r_n(x)^{\theta}} \geq \frac{C_0}{2 \delta^\theta} - L r_n(x)^{1 - \theta}.
\]
Since $r_n(x) / r_{n+1}(x) < C_1$ is uniformly bounded, for any $x \in M$ and $R > 0$ small, there exists a $n_R \in \mathbb{Z}^+$ such that $R/C_1 \leq r_{n_R}(x) \leq R$, and
\[
h_{\theta, R}(x)= \frac{\text{osc}(\Phi, x, R)}{R^\theta} \geq \frac{d(\Phi(t_{n_R}), \Phi(x))}{(C_1 r_{n_R}(x))^\theta} \geq C_1^{-\theta} \left( \frac{C_0}{2 \delta^\theta} - R^{1 - \theta} \right).
\]
Therefore, there exists a small constant $R_0 \in \mathbb{R}^+$ such that for every $x \in M$ and $R < R_0$,
\[
h_{\theta, R}(x) \geq \frac{C_0}{4 \delta^\theta C_1^\theta}.
\]
By Lemma \ref{lemma dim criterion}, we conclude that $\underline{\dim}_B \mathrm{Graph}(\Phi) \geq \dim M + 1 - \theta$. Since $\theta$ can be chosen arbitrarily close to $A_1$, the proof is complete.
\end{proof}

\section{Proof of the remaining results}\label{sec: proof remain}
In this section, we demonstrate how to apply Theorem \ref{theorem cocycle pointwise} to study the invariant distributions of partially hyperbolic Anosov diffeomorphisms.

\subsection{Proof of Theorem \ref{theorem s bunching}}
By the bunching assumption, $E^{cs}$ is $C^{1}$ \cite{has95}. Let $N = \mathrm{Gr}(\dim E^s, E^{cs})$ be the Grassmannian bundle consisting of subspaces of the same dimension as $E^s$ in $E^{cs}$. Then, $D f|_{E^{cs}}$ induces a natural bundle map $F: N \to N$, and $E^s$ gives an $F$-invariant section in $N$. A direct calculation illustrates:
\begin{lemma}\label{lemma ph check}
    $F$ is partially hyperbolic and $\alpha = \alpha_s$.
\end{lemma}
\begin{proof}
    We first check that $F$ is partially hyperbolic. By assumption, we have
    \[
    \frac{\|Df^k|_{E^{c}(x)}\|}{m(Df^k|_{E^{s}(x)})} <  m(Df^k|_{E^u(x)}), \quad \forall x \in M.
    \]
    By triangle inequality, for any $l \in \ZZ^+$ and $x \in M$,
    \[
    \frac{\|Df^{lk}|_{E^{c}(x)}\|}{m(Df^{lk}|_{E^{s}(x)})} \leq \prod_{i=0}^{l-1} \frac{\|Df^k|_{E^{c}(f^{ik}(x))}\|}{m(Df^k|_{E^{s}(f^{ik}(x))})} <\prod_{i=0}^{l-1} m(Df^k|_{E^u(f^{ik}(x))}) \leq \, m(Df^{lk}|_{E^u(x)}).
    \]
    For any $\epsilon > 0$, by choosing $l$ sufficiently large, we have:
    \[
    \|Df^{lk}|_{E^{cs}(x)}\| \leq \sqrt{1 + \epsilon} \, \|Df^{lk}|_{E^{c}(x)}\|, \quad m(Df^{lk}|_{E^{cs}(x)})^{-1} \leq \sqrt{1 + \epsilon} \, m(Df^{lk}|_{E^{s}(x)})^{-1}.
    \]
    Therefore, for $\epsilon$ sufficiently small, we have:
    \[
    \|F^{lk}(x)\| = \frac{\|Df^{lk}|_{E^{cs}(x)}\|}{m(Df^{lk}|_{E^{cs}(x)})} \leq (1 + \epsilon) \frac{\|Df^{lk}|_{E^{c}(x)}\|}{m(Df^{lk}|_{E^{s}(x)})} < m(Df^{lk}|_{E^u(x)}), \quad \forall x \in M,
    \]
    i.e., $F$ is partially hyperbolic.

    We then show that $\alpha = \alpha_s$. Note that
    \[
    \|F^{n}(x)\| = \frac{\|Df^{n}|_{E^{cs}(x)}\|}{m(Df^{n}|_{E^{cs}(x)})} \geq \frac{\|Df^{n}|_{E^{c}(x)}\|}{m(Df^{n}|_{E^{s}(x)})}, \quad \forall n \in \ZZ, \, x \in M,
    \]
    we have $\alpha(k, x) \geq \alpha_s(k, x)$ and $\alpha \geq \alpha_s$. Now, for any $\epsilon > 0$, there exists $k \in \ZZ^+$ and $x \in M$ such that
    \[
    \alpha_s(k, x) \leq \alpha_s + \epsilon.
    \]
    By definition, we have:
    \[
    \alpha_s(lk, x) \leq \alpha_s(k, x), \quad \forall l \in \ZZ^+.
    \]
    By the estimate above, for $l$ sufficiently large, we have:
    \[
    \|F^{lk}(x)\| \leq (1 + \epsilon) \frac{\|Df^{lk}|_{E^{c}(x)}\|}{m(Df^{lk}|_{E^{s}(x)})},
    \]
    which implies
    \[
    \alpha(lk, x) \leq (1 + \epsilon) \alpha_s(lk, x).
    \]
    Combining these facts together, we have
    \[
    \alpha \leq \alpha(lk, x) \leq (1 + \epsilon)(\alpha_s + \epsilon).
    \]
    This finishes the proof since $\epsilon$ can be taken arbitrarily small.
\end{proof}

Since $W^u$ is minimal, by Theorem \ref{theorem cocycle pointwise}, we obtain:
\begin{enumerate}
    \item If $E^s$ is uniformly $\alpha_s+$-H\"older continuous along $W^u$, then $E^s$ is invariant under $h^u_F$ and is uniformly $C^{1}$ along $W^u$, 
    \item If (1) does not hold, then $\underline \dim_{B} \mathrm{Graph}(E^s) > \dim M$.
\end{enumerate}

Since $W^s$ is a $C^2$ subfoliation of $W^{cs}$ \cite{hps77}, $E^s$ is uniformly $C^{1}$ along $W^{cs}$. By Journ\'{e}'s Lemma \cite{jou88}, we have:
\begin{corollary}\label{corollary es c1}
If $E^s$ is uniformly $\alpha_s+$-H\"older continuous along $W^u$, then $E^s$ is $C^{1}$.
\end{corollary}

We now prove that if $E^s$ is $C^{1}$, then $E^s \oplus E^u$ is jointly integrable. Since $E^u$ is bunching, $E^u$ is also $C^{1}$ \cite{has95}, and the unstable holonomy $h^u_{xy}: W^{cs}_\text{loc}(x) \to W^{cs}_\text{loc}(y)$ of $W^u_f$ is uniformly $C^{1}$ \cite{psw97}. We denote its derivative by $D h^u_{xy}: E^{cs}(x) \to E^{cs}(y)$.\begin{lemma}\label{lemma lip holonomy}
   If $E^{cs}$ and $E^u$ are $C^{1}$, then $D h^u_{xy}$ is uniformly Lipschitz with respect to $x$ and $y$, i.e., there exists a constant $C$ such that $\|D h^u_{xy}(x) - \text{Id}\| \leq C d(x, y)$ for every $x \in M$ and $y \in W^u_\text{loc}(x)$ sufficiently close to $x$.
\end{lemma}

\begin{proof}
To simplify the notation, we first prove the lemma for $\dim E^u = 1$. Since the lemma is local and invariant under smooth coordinate changes, we can reformulate the problem on the manifold $\mathbb{R}^{m} = \mathbb{R}^{m-1} \times \mathbb{R}$ with $x = (\boldsymbol{0}, 0)$, $W^u(x) = \{\boldsymbol{0}\} \times \mathbb{R}$, and $W^{cs}(x) = \mathbb{R}^{m-1} \times \{0\}$. For $y = (\boldsymbol{0}, \xi)$ and $z = (\boldsymbol{\zeta}, 0)$, with $\|\boldsymbol{\zeta}\| \ll \xi \ll 1$ sufficiently small (which will be determined later), denote $w := h^u_{xy}(z) = (\boldsymbol{\omega_1}, \omega_2)$.

We first estimate $\frac{\partial \boldsymbol{\omega_1}}{\partial \boldsymbol{\zeta}}$ at $\boldsymbol0$. Let 
$$z(t) = (\boldsymbol{\zeta}(t), t) := W^u_\text{loc}(z) \cap (\mathbb{R}^{m-1} \times \{t\}),$$
and consider the function $\varphi_z(t) := \|\boldsymbol{\zeta}(t) - \boldsymbol{\zeta}\|$. Clearly, $\varphi_z(0) = 0$. For any $p = (\boldsymbol{p}_1, p_2) \in \mathbb{R}^m$, let $(\boldsymbol{e}(p), 1)$ be a vector tangent to $E^u$. Since $E^u$ is $C^{1}$ and $\boldsymbol{e}|_{W^u(x)} = 0$, there exists a uniform constant $L$ such that 
$$\|\boldsymbol{e}(p)\| \leq L\|\boldsymbol{p}_1\|.$$
This implies
$$\left|\frac{d \varphi_z(t)}{d t}\right| \leq \|\boldsymbol{e}(z(t))\| \leq L\|\boldsymbol{\zeta}(t)\| \leq L(\|\boldsymbol{\zeta}\| + \varphi_z(t)).$$
By Gronwall's inequality, for $\xi < 1/(10L)$, we have 
$$\|\boldsymbol{\zeta}(t) - \boldsymbol{\zeta}\| = \varphi_z(t) \leq \|\boldsymbol{\zeta}\|(e^{Lt} - 1) \leq 2Lt\|\boldsymbol{\zeta}\|.$$
Since $\boldsymbol{\omega_1} = \boldsymbol{\zeta}(\omega_2)$ and $\omega_2 \to \xi$ as $\|\boldsymbol{\zeta}\| \to 0$, we obtain:
$$\left\|\frac{\partial \boldsymbol{\omega_1}}{\partial \boldsymbol{\zeta}}(\boldsymbol0) - \text{Id}\right\| = \left\|\frac{\partial (\boldsymbol{\omega_1} - \boldsymbol{\zeta})}{\partial \boldsymbol{\zeta}}(\boldsymbol0)\right\| = \limsup_{\boldsymbol{\zeta} \to \boldsymbol0} \frac{\|\boldsymbol{\zeta}(\omega_2) - \boldsymbol{\zeta}\|}{\|\boldsymbol{\zeta}-\boldsymbol{0}\|} \leq 2L\xi.$$

Note that $D h^u_{xy}(x)$ maps $E^{cs}(x)$ to $E^{cs}(y)$. Since $E^{cs}$ is $C^{1}$, the angle between $E^{cs}(y)$ and $\RR^{n-1} \times \{\xi\}$ is bounded by $L'\xi$ for some constant $L'$, and we have
$$\|D h^u_{xy}(x) - \text{Id}\| \leq  \sec(L'\xi)\left\|\frac{\partial \boldsymbol{\omega_1}}{\partial \boldsymbol{\zeta}}-\text{Id} \right\| \leq 2L\xi+O(\xi^2)< 3L \xi$$
for $\xi = d(x, y)$ small. This completes the proof for the case $\dim E^u = 1$.

For the general case, note that $h^u_{xy}$ is independent of the path in $W^{u}(x)$ joining $x$ and $y$. We can reformulate the problem on the manifold $\mathbb{R}^{cs+u} = \mathbb{R}^{cs+u-1} \times \mathbb{R}$ with $x = (\boldsymbol{0}, 0)$, $y = (\boldsymbol{0}, \xi)$, and $\{\boldsymbol{0}\} \times \mathbb{R} \subset W^u(x)$, $W^{cs}(x) \subset \mathbb{R}^{cs+u-1} \times \{0\}$. Then we can apply the arguments above on the $cs$-saturated set $S := \bigcup_{t \in \mathbb{R}} W^{cs}((\boldsymbol{0}, t)) \subset \mathbb R^{cs+u}$ to prove the lemma.
\end{proof}
\begin{proposition}\label{prop su int c1+}
If $E^s$ is uniformly $C^{1}$ along $W^u$, then $E^s \oplus E^u$ is jointly integrable. 
\end{proposition}
\begin{proof}
Denote $\Phi^s: M \to N$ as the $F$-invariant section induced by $E^s$. Since $E^s$ is uniformly $C^{1}$ along $W^u$, $\Phi^s$ is uniformly Lipschitz along $W^u$. Recall that $F$ is partially hyperbolic; there exist constants $\epsilon > 0$ small and $k \in \mathbb{Z}^+$ such that
$$\|F^k_x\| \leq m (D f^k|_{E^u(x)})^{1-2\epsilon}, \quad \forall x \in M.$$
Choose $\delta > 0$ such that Lemma \ref{lemma bounded type} holds. For any $x \in M$ and $y \in W^{u}_\delta(x)$, the derivative of the unstable holonomy $D h^u_{xy}: E^{cs}(x) \to E^{cs}(y)$ induces a map $H^u_{xy}: N_x \to N_y$. Define the map $\tilde{\gamma}_x: W^u_\text{loc}(x) \to N$ by $\tilde{\gamma}_x(y) := H^u_{xy}(\Phi^s(x))$. By Lemma \ref{lemma lip holonomy}, $\tilde{\gamma}_x$ is uniformly Lipschitz.

Now, for $x \in M$ and $y \in W^u_\text{loc}(x)$, similar to the estimation in Lemma \ref{lemma exce holder}, we have:
\begin{align*}
d(\Phi^s(y), \tilde{\gamma}_x(y)) & = d(F^{nk}(\Phi^s(f^{-nk}y)), F^{nk}(\tilde{\gamma}_{f^{-nk}x}(f^{-nk}y))) \\
& \leq \|F^{nk}_{f^{-nk}y}\| \cdot d(\Phi^s(f^{-nk}y), \tilde{\gamma}_{f^{-nk}x}(f^{-nk}y)) \\
& \leq \|F^{nk}_{f^{-nk}y}\| \cdot \big(d(\Phi^s(f^{-nk}y), \Phi^s(f^{-nk}x)) + d(\tilde{\gamma}_{f^{-nk}x}(f^{-nk}x), \tilde{\gamma}_{f^{-nk}x}(f^{-nk}y))\big) \\
& \leq C \|F^{nk}_{f^{-nk}y}\| \cdot d(f^{-nk}x, f^{-nk}y),
\end{align*}
for some constant $C>0$. By Lemma \ref{lemma bounded type} and the partial hyperbolicity assumption of $F$, we know
$$\|F^{nk}_{f^{-nk}y}\| \leq \prod_{i=1}^n \|F^k_{f^{-ik}y}\| \leq \prod_{i=1}^n \|F^k_{f^{-ik}x}\|^{1+\epsilon} \leq \prod_{i=1}^n m (D f^k|_{E^u(f^{-ik}x)})^{(1-2\epsilon)(1+\epsilon)},$$
and
$$d(f^{-nk}x, f^{-nk}y) \leq \delta \prod_{i=1}^n m (D f^k|_{E^u(f^{-ik}x)})^{-1+\epsilon}.$$
Therefore, combining these estimates:
$$d(\Phi^s(y), \tilde{\gamma}_x(y)) \leq C\delta \prod_{i=1}^n m (D f^k|_{E^u(f^{-ik}x)})^{-2\epsilon^2} \to 0 \quad \text{as } n \to \infty.$$
This implies $\Phi^s(y) = \tilde{\gamma}_x(y) = H^u_{xy}(\Phi^s(x))$. In other words, $E^s(y) = D h^u_{xy}(E^s(x))$, so $E^s$ is invariant under $D h^u$. Thus, $h^u_{xy}(W^s_\text{loc}(x))$ is a $C^{1}$ submanifold tangent to $E^s$ everywhere. By the unique integrability of $E^s$, $h^u_{xy}(W^s_\text{loc}(x))$ is a local strong stable leaf. Hence, $h^u$ maps $W^s$ leaves to $W^s$ leaves, and $E^s \oplus E^u$ is jointly integrable.
\end{proof}
\subsection{Proof of Theorem \ref{theorem 3dim s} and Corollary \ref{coro: local qualitative dich}}
Let $f \in \Diff^{2}(\mathbb{T}^3)$ be $C^1$-close to a volume-preserving diffeomorphism $g$, then $f$ satisfies all the conditions in Theorem \ref{theorem s bunching}. Therefore, by applying Theorem \ref{theorem s bunching} to $f$, along with Proposition \ref{prop 3dim pinch}, the result in Theorem \ref{theorem 3dim s} follows immediately. Moreover, since $E^s \oplus E^u$ is jointly integrable if and only if if $f$ has the same $c$-periodic data as its linearization $L$ \cite{gs20}, we see that such $f$ form a codimension-\( \infty \) subset in $\Diff_{\mathrm{vol}}^{2}(\mathbb{T}^3)$. 
\subsection{Proof of Theorem \ref{theorem c bunching ss}}
By the bunching assumption, $E^{cs}$ is $C^{1}$ \cite{has95}. Let $N = \mathrm{Gr}(\dim E^c, E^{cs})$ denote the Grassmannian bundle consisting of subspaces of $E^{cs}$ with the same dimension as $E^c$. The derivative $D f^{-1}|_{E^{cs}}$ induces a natural bundle map $F: N \to N$ over $f^{-1}:M \to M$. Similar to Lemma \ref{lemma ph check}, we can check that $F$ is partially hyperbolic and $\alpha=\alpha_c$ in this case. By Applying Theorem \ref{theorem cocycle pointwise} to the cocycle defined by $F$, we conclude the proof.

\subsection{Proof of Theorem \ref{theorem 3dim c}}
By Theorem \ref{theorem c bunching ss} and Proposition \ref{prop 3dim pinch}, we establish the following two cases:
\begin{enumerate}
    \item If $E^c$ is uniformly $\theta_c+$-H\"older continuous along $W^s$, then $E^c$ is invariant under $h^s_F$ and is uniformly $C^{1}$ along $W^s$.
    \item If (1) does not hold, then $\underline{\dim}_{B} \mathrm{Graph}(E^c) > \dim M$.
\end{enumerate}
We then prove that $W^c$ is a $C^{1}$ foliation when $E^c$ is uniformly $C^{1}$ along $W^s$. 

Before proceeding to the proof, we state the following lemma from real analysis, which will be utilized in the argument:
\begin{lemma}\label{lemma lip c1}
    Let $U \subset \mathbb{R}$ be an open set, and let $f: U \to \mathbb{R}$ be a Lipschitz continuous function. Then $f$ is differentiable at Lebesgue-a.e. $ x \in U$. Moreover, if $Df$ is H\"older continuous on a set of full Lebesgue measure, then $f$ is $C^{1+}$. 
\end{lemma}
\begin{proposition}\label{prop ec imply wc}
If $E^c$ is uniformly $C^{1}$ along $W^s$, then $W^c$ is a $C^{1}$ foliation.
\end{proposition}
\begin{proof}
Since $f$ is dynamically coherent \cite{bbi09,por15} and the center-stable holonomy $h^{cs}$ is uniformly $C^{1}$ \cite{psw97}, the leaves of $W^c$ are $C^{1}$ and the center holonomy is uniformly $C^{1}$ along $W^u$. By Journ\'{e}'s lemma \cite{jou88}, it suffices to prove that the center holonomy along $W^s$, $h^c_{xy}: W^s_\text{loc}(x) \to W^s_\text{loc}(y)$, is uniformly $C^{1}$.

We first prove that $h^c$ is uniformly Lipschitz. The proof is similar to that of Lemma \ref{lemma lip holonomy}. Since $W^s$ is a $C^2$ subfoliation of $W^{cs}$ \cite{hps77}, we can reduce the problem to $\mathbb{R}^2$ via a $C^2$ coordinate change such that $x = (0, 0)$, $W^c(x)=\{0\} \times \RR$ and the stable leaves are parallel to the $x$-axis. For $p \in \mathbb{R}^2$, let $ (e(p), 1)$ be a vector tangent to $E^c(p)$. By assumption, $e(p)$ is uniformly Lipschitz along $W^s$. For $z = (\zeta, 0) \in W^s_\text{loc}(x)$ and $t \in \RR$, let $y(t) := (0, t)$ and $z(t) := h^c_{xy(t)}(z) = (\zeta(t), t)$. Then we have $\zeta(0) = \zeta$ and
\[
\left| \frac{d \zeta(t)}{d t} \right| = |e(z(t)) - e(y(t))| \leq L |\zeta(t)|
\]
for some constant $L$. By Gronwall's inequality, for $0 \leq |t| < 1/(10L)$, we have:
\[
\|h^c_{xy(t)}(z) - y(t)\| = |\zeta(t)| \leq e^{L|t|} |\zeta| \leq (1 + 2L|t|) |\zeta|.
\]
Thus, $h^c$ is uniformly Lipschitz and is is differentiable almost everywhere. Similarly, we have $ |\zeta(t)| \geq (1-2L|t|)|\xi|$ and if $h^c_{xy(t)}$ is differentiable at $x$, then
$$|D h^c_{xy(t)}(x)-1 | = \left|\lim_{\zeta \to 0}\frac{\|h^c_{xy(t)}(z) - y(t)\|}{|\zeta|}-1\right| \leq  2L|t|=2Ld(x,y(t)).$$
We now show the derivative of $h^c$ is H\"older continuous. For $z_1, z_2 \in W^s_\text{loc}(x)$ such that $h^c_{xy}$ is differentiable at both points, let $w_1 = h^c_{xy}(z_1)$ and $w_2 = h^c_{xy}(z_2)$. Differentiating the relation 
\[
h^c_{xy} = f^{-n} \circ h^c_{f^n x, f^n y} \circ f^n,
\]
we obtain:
\[
D h^c_{xy}(z_i) = (D f^n|_{E^s(w_i)})^{-1} \cdot D h^c_{f^n x, f^n y}(f^n z_i) \cdot D f^n|_{E^s(z_i)}, \quad i = 1, 2.
\]
By the classical bounded distortion estimates (e.g., \cite{kh95}), for $z_1$ and $z_2$ sufficiently close:
\[
\lim_{n \to \infty} \frac{|D f^n|_{E^s(z_2)}|}{|D f^n|_{E^s(z_1)}|} \leq e^{Cd(z_1, z_2)^\beta} \leq 1 + 2Cd(z_1, z_2)^\beta,
\]
where $\beta$ is the H\"older exponent of $E^s$, and $C$ is a uniform constant. Since $|D h^c_{xy}(x) - 1| < 2L d(x, y)$ for $y$ near $x$, we have:
\begin{align*}
\left| \frac{D h^c_{xy}(z_2)}{D h^c_{xy}(z_1)} \right| & = \lim_{n \to \infty} \left| \frac{(D f^n|_{E^s(w_2)})^{-1} \cdot D h^c_{f^n x, f^n y}(f^n z_2) \cdot D f^n|_{E^s(z_2)}}{(D f^n|_{E^s(w_1)})^{-1} \cdot D h^c_{f^n x, f^n y}(f^n z_1) \cdot D f^n|_{E^s(z_1)}} \right|  \\
& \leq  (1 + 2Cd(w_1, w_2)^\beta)(1 + 2Cd(z_1, z_2)^\beta).
\end{align*}
This implies that $D h^c_{xy}$ is H\"older continuous on a subset of full Lebesgue measure in $W^s_\text{loc}(x)$. By Lemma \ref{lemma lip c1}, $h^c$ is $C^{1+}$.
\end{proof}

\subsection{Proof of Corollary \ref{theorem 3dim c local}}
We first show that \(E^c\) is uniformly \(C^{1}\) along \(W^u\).
\begin{lemma}\label{lemma c along u}
   \(E^c\) is uniformly \(C^{1}\) along \(W^u\). 
\end{lemma}
\begin{proof}
Recall that $E^{cs}$ is $C^1$ \cite{has95}. Let \(N = \mathrm{Gr}(\dim E^s, E^{cs})\). Then \(D f|_{E^{cs}}\) induces a partially hyperbolic bundle map \(F: N \to N\), and \(E^c\) defines an \(F\)-invariant section \(\Phi^c: M \to N\). For any \(x \in M\), let 
\[
\gamma_x: W^u_\text{loc}(x) \to N, \quad y \mapsto (h^u_F)_{xy}(\Phi^c(x)).
\]
Note that \(\mathrm{Graph}(\Phi^c)\) is an attractor of \(F\). For every \(t \in W^u_\text{loc}(x)\) and \(n \in \mathbb{Z}^+\), we have
\[
d(\Phi^c(t), \gamma_x(t)) = d(F^{n}(\Phi^c(f^{-n}t)), F^{n}(\gamma_{f^{-n}x}(f^{-n}t))) < d(\Phi^c(f^{-n}t), \gamma_{f^{-n}x}(f^{-n}t)) \to 0
\]
as \(n \to \infty\). Therefore, \(\Phi^c|_{W^u(x)} = \gamma_x\) coincides with an unstable leaf of \(W^u_F\), and \(E^c\) is uniformly \(C^{1}\) along \(W^u_f\). 
\end{proof}

By Journ\'{e}'s lemma \cite{jou88} and Lemma \ref{lemma c along u}, it suffices to prove that \(E^c\) is uniformly \(C^{1}\) along \(W^{cs}\) when \(E^c\) is not fractal. Note that the characteristic polynomial of \(L : \mathbb{T}^3 \to \mathbb{T}^3\) is irreducible, and the Galois group of the characteristic polynomial acts transitively on the eigenvalues \(\{\lambda^s_L, \lambda^c_L, \lambda^u_L\}\) of $L$. In particular, \(\kappa := \frac{\log \lambda^s_L}{\log \lambda^c_L} \notin \mathbb{Z}\). We then prove Theorem \ref{theorem 3dim c local} for different values of \(\kappa\): 

\begin{proposition}
  If \(f\) is \(C^1\) close to \(L\) with \(\kappa > 2\) and \(E^c\) is not fractal, then \(E^c\) is uniformly \(C^{1}\) along \(W^{cs}\).
\end{proposition}
\begin{proof}
By Theorem \ref{theorem 3dim c}, if \(E^c\) is not fractal, then \(E^c\) is uniformly \(C^{1}\) along \(W^s\). Since \(\kappa > 2\), the leaves of \(W^c\) are uniformly \(C^{2}\) (see Proposition \ref{prop anosov t3}), thus \(E^c\) is always uniformly \(C^{1}\) along \(W^c\). By Journ\'{e}'s lemma, \(E^c\) is \(C^{1}\) along \(W^{cs}\).
\end{proof}
\begin{proposition}
  If \(f\) is \(C^1\) close to \(L\) with \(\kappa < 2\) and \(E^c\) is not fractal, then \(E^c\) is uniformly \(C^{1}\) along \(W^{cs}\).
\end{proposition}
\begin{proof}
Let \(N = \PP(E^{cs})\) and \(F: N \to N\) be the bundle map induced by \(D f^{-1}|_{E^{cs}}\) over $f^{-1}$. Since \(\kappa < 2\) and \(f\) is \(C^1\) close to \(L\), for any $x \in M$, $\|F_x\|$ is close to $\frac{|D f^{-1}|_{E^{s}(x)}|}{|D f^{-1}|_{E^{c}(x)}|}< |D f^{-1}|_{E^{c}(x)}|$. Therefore, \(F\) is partially hyperbolic with respect to \(E^u_{f^{-1}} = E^{cs}_f\). Applying Theorem \ref{theorem cocycle pointwise} to \(F\) and \(W^u_{f^{-1}} = W^{cs}_f\), we conclude that \(E^c\) is uniformly \(C^{1}\) along \(W^{cs}\) if \(E^c\) is not fractal.
\end{proof}
\subsection{Proof of Theorem \ref{theorem 3dim rigid}}
Let $L: \mathbb{T}^3 \to \mathbb{T}^3$ be an Anosov automorphism homotopic to $f$. Since $E^s$ is $\theta_s+$-H\"older continuous, Theorem \ref{theorem 3dim s} implies that $E^s \oplus E^u$ is jointly integrable. This ensures that $f$ and $L$ share the same $c$-periodic data \cite{gs20}, and that $W^s$ is minimal \cite{rgz17}. 

Similarly, by Theorem \ref{theorem 3dim c}, the $\theta_c+$-H\"older continuity of $E^c$ ensures that $W^c$ is a $C^{1}$ foliation. This implies that $f$ and $L$ share the same $s$-periodic data \cite{gog12}.

Finally, note that $f$ is volume-preserving. For any periodic point $p$, we have the relation $\lambda^u(p) = -(\lambda^c(p) + \lambda^s(p))$ among the Lyapunov exponents. Consequently, $f$ and $L$ have the same periodic data. By the rigidity results in \cite{gg08, gog17}, $f$ is $C^\infty$ conjugate to $L$.

We end this section with examples illustrating that the volume-preserving condition and the three-dimensional assumption are necessary in Theorem \ref{theorem 3dim rigid}. 
\begin{example}\label{example c1}
Let \(L:\mathbb T^3 \to \mathbb T^3\) be a partially hyperbolic linear Anosov automorphism with \(E^c_L\) being contracting. Let \(\textbf{e}\) be the unit vector tangent to \(E^u_L\), and let \(\varphi:\mathbb T^3 \to \mathbb R\) be a \(C^\infty\) function such that \(\varphi(0) = 0\) and \(\frac{\partial \varphi}{\partial \textbf{e}}(0) \neq 0\). For any small \(\epsilon > 0\), define
    \[
    f : \mathbb T^3 \to \mathbb T^3, \quad x \mapsto Lx + \epsilon \varphi(x) \textbf{e}.
    \]
    Since \(f\) is \(C^1\)-close to \(L\), \(f\) is Anosov, and the stable and unstable distributions \(E^{cs}\) and \(E^u\) are \(C^{1}\) distributions \cite{has97}. Note that
    \[
    W^u_f = W^u_L, \quad W^{cu}_f = W^{cu}_L, \quad W^{su}_f = W^{su}_L
    \]
are \(C^\infty\) foliations. In particular, \(E^u\) is \(C^\infty\), and \(E^c = E^{cu} \cap E^{cs}\) is \(C^{1}\). Since $W^{su}_f$ is well-defined, by the bunching arguments \cite{psw97} and Journ\'e's Lemma \cite{jou88}, $W^s$ is a $C^1$ foliation, and $E^s$ is $1-$-H\"older continuous \cite{hw99}. By Theorem \ref{theorem 3dim s}, $E^s$ is $C^1$. However, we have: 
    \[
    \lambda^u_f(0) = \lambda^u_L(0) + \frac{\partial \varphi}{\partial \textbf{e}}(0) \neq \lambda^u_L(0),
    \]
 thus \(f\) is not \(C^1\)-conjugate to \(L\). 
\end{example}
\begin{example}\label{example dim}
    Let \(f: \mathbb T^2 \to \mathbb T^2\) be a volume-preserving Anosov automorphism that is not \(C^1\)-conjugate to the linear model $L_f$. For any \(n \geq 2\), let \(L: \mathbb T^n \to \mathbb T^n\) be a linear Anosov automorphism. Then \(f \times L : \mathbb T^{n+2} \to \mathbb T^{n+2}\) is a volume-preserving Anosov diffeomorphism with all distributions \(C^1\), but \(f \times L\) is not \(C^1\)-conjugate to \(L_f \times L\).
\end{example}

\appendix

\section{Pinching coefficients and critical H\"older exponents}
The aim of this appendix is to show that the critical H\"older exponents \(\alpha_s\) and \(\alpha_c\) in Theorem \ref{theorem s bunching} and Theorem \ref{theorem c bunching ss} coincide with the supremum of the pinching coefficients \(\theta_s\) and \(\theta_c\) respectively when \(M = \TT^3\). Consequently, we can use \(\theta_s\) and \(\theta_c\) as the critical H\"older exponents in Theorem \ref{theorem 3dim s} and Theorem \ref{theorem 3dim c}. The following lemma is a special case of a Liv\v sic type theorem for matrix valued cocycle:

\begin{lemma}[Theorem 1.4 in \cite{kal11}]\label{lemma livsic}
    Let \(f:M \to M\) be a topologically transitive Anosov diffeomorphism, and let \(A:M \to \mathbb R\) be an \(\mathbb R\)-valued cocycle. If \(A^n(p) > e^{n\alpha} > 1\) for every \(p \in M\) with \(f^n p = p\), then for any \(\epsilon > 0\), there exists a constant \(C_\epsilon\) such that 
    \[
    A^k(x) \geq C_\epsilon e^{k(\alpha - \epsilon)}, \quad \forall x \in M, \, k \in \mathbb{N}.
    \]
\end{lemma}
\begin{proposition}\label{prop 3dim pinch}
Let \(f \in \Diff^{2}(\mathbb T^3)\) be a partially hyperbolic Anosov diffeomorphism. Then \(\theta_s = \alpha_s\) and \(\theta_c = \alpha_c\). 
\end{proposition}
\begin{proof}
We prove \(\theta_s = \alpha_s\); the proof for \(\theta_c = \alpha_c\) is similar. 

We first prove \(\theta_s \leq \alpha_s\). It follows from the definition of \(\theta_s\) that for any \(\theta < \theta_s\), there exists a constant \(k_\theta \in \mathbb Z^+\) such that
\[
|D f^{k_\theta}|_{E^s(x)}| \cdot |D f^{k_\theta}|_{E^u(x)}|^\theta < |D f^{k_\theta}|_{E^c(x)}|, \quad \forall x \in \TT^3.
\]
Since \(\dim E^u = \dim E^c = \dim E^s = 1\),
\[
\alpha_s \geq \alpha_s(x, k_\theta) = \liminf_{n \to +\infty} \frac{\sum_{i=1}^n \log |D f^{k_\theta}|_{E^c(f^{-ik_\theta}x)}| - \log |D f^{k_\theta}|_{E^s(f^{-ik_\theta}x)}|}{\sum_{i=1}^n \log |D f^{k_\theta}|_{E^u(f^{-ik_\theta}x)}|} \geq \theta,
\]
which implies \(\theta_s \leq \alpha_s\). 

We then prove \(\theta_s \geq \alpha_s\). For any \(\alpha < \alpha_s\), consider the \(\mathbb R\)-valued cocycle
\[
A_\alpha(x) := \frac{|Df|_{E^c(x)}|}{|Df|_{E^s(x)}| \cdot |Df|_{E^u(x)}|^\alpha}.
\]
By taking a finite iteration, we may assume that \(|Df|_{E^u(x)}| > e\), \(\forall x \in \TT^3\). Then for every \(p \in M\) with \(f^l p = p\),
\[
A_\alpha^l(p) = |Df^l|_{E^u(p)}|^{\alpha_s(p, l) - \alpha} \geq e^{l(\alpha_s - \alpha)} > 1.
\]
By Lemma \ref{lemma livsic}, for every \(\epsilon > 0\), there exists a constant \(C_\epsilon > 0\) such that 
\[
A_\alpha^k(x) \geq C_\epsilon e^{k(\alpha_s - \alpha - \epsilon)},
\]
for every \(k \in \mathbb Z^+\) and \(x \in M\). By choosing \(\epsilon < \alpha_s - \alpha\) small and \(k_0 \in \mathbb Z^+\) large, we have \(A^{k_0}_\alpha(x) > 1\) for every \(x \in M\), i.e., \(E^s\) is \(\alpha\)-pinching, which implies \(\theta_s \geq \alpha_s\) since \(\alpha\) can be taken arbitrarily close to \(\alpha_s\). 
\end{proof}
We then recall some facts about pinching and H\"older continuity. For any \(f \in \Diff^1(M)\), an \(f\)-invariant decomposition \(T M = E \oplus F\) is called a \textit{dominated splitting} if:  
\begin{itemize}
    \item \(E\) and \(F\) are \(D f\)-invariant;
    \item there exist constants \(k \in \mathbb Z^+\) and \(0 < \lambda < 1\) such that \(\|D f^k|_{E(x)}\| \leq \lambda^k m (D f^k|_{F(x)})\) for every \(x \in M\). 
\end{itemize}
\begin{proposition}[Theorem 4.1 in \cite{cp15}]\label{prop pinching}
For any \(f \in \Diff^2(M)\) with dominated splitting \(T M = E \oplus F\), if \(E\) is \textit{\(\theta\)-pinching} in the sense that
\[
\|D f^k|_{E(x)}\| \cdot \|D f^k|_{F(x)}\|^\theta < m(D f^k|_{F(x)}), \quad \forall x \in M,
\]
then \(E\) is \(\theta\)-H\"older continuous. 
\end{proposition}
It is straightforward to check that \(\theta_s\) and \(\theta_c\), as defined in Section \ref{section 3dim}, are the supremum of the pinching coefficients of \(E^s\) and \(E^{cu}\) (with respect to \(f^{-1}\)) respectively. Note that \(E^{cs}\) is \(C^{1}\), and \(E^c = E^{cs} \cap E^{cu}\) has at least the same H\"older regularity as \(E^{cu}\). Therefore, we have:
\begin{corollary}\label{coro pinching}
For any partially hyperbolic Anosov diffeomorphism \(f \in \Diff^{2}(\mathbb T^3)\), \(E^s\) is \(\theta_s\)-H\"older continuous, and \(E^c\) is \(\theta_c\)-H\"older continuous. 
\end{corollary}


\bibliographystyle{plain}
\bibliography{ref}

\begin{thebibliography}{10}

\bibitem{alos24}
S.~Alvarez, M.~Leguil, D.~Obata, and B.~Santiago.
\newblock Rigidity of {$u$}-gibbs measures near conservative {A}nosov diffeomorphisms on {$\Bbb{T}^3$}.
\newblock {\em J. Eur. Math. Soc. (JEMS)}, published online first, 2024.

\bibitem{Ano69}
D.~V. Anosov.
\newblock Geodesic flows on closed {R}iemannian manifolds of negative curvature.
\newblock {\em Trudy Mat. Inst. Steklov.}, 90:209, 1967.

\bibitem{av68}
A.~Avez.
\newblock Anosov diffeomorphisms.
\newblock In {\em Topological {D}ynamics ({S}ymposium, {C}olorado {S}tate {U}niv., {F}t. {C}ollins, {C}olo., 1967)}, pages 17--51. W. A. Benjamin, Inc., New York-Amsterdam, 1968.

\bibitem{avw1}
A.~Avila, M.~Viana, and A.~Wilkinson.
\newblock Absolute continuity, {L}yapunov exponents and rigidity {I}: geodesic flows.
\newblock {\em J. Eur. Math. Soc. (JEMS)}, 17(6):1435--1462, 2015.

\bibitem{avw2}
A.~Avila, M.~Viana, and A.~Wilkinson.
\newblock Absolute continuity, {L}yapunov exponents, and rigidity {II}: systems with compact center leaves.
\newblock {\em Ergodic Theory Dynam. Systems}, 42(2):437--490, 2022.

\bibitem{BHR}
B.~B\'ar\'any, M.~Hochman, and A.~Rapaport.
\newblock Hausdorff dimension of planar self-affine sets and measures.
\newblock {\em Invent. Math.}, 216(3):601--659, 2019.

\bibitem{bed89}
T.~Bedford.
\newblock The box dimension of self-affine graphs and repellers.
\newblock {\em Nonlinearity}, 2(1):53--71, 1989.

\bibitem{bl93}
Y.~Benoist and F.~Labourie.
\newblock Sur les diff\'eomorphismes d'{A}nosov affines \`a\ feuilletages stable et instable diff\'erentiables.
\newblock {\em Invent. Math.}, 111(2):285--308, 1993.

\bibitem{bow79}
R.~Bowen.
\newblock Hausdorff dimension of quasicircles.
\newblock {\em Inst. Hautes \'Etudes Sci. Publ. Math.}, (50):11--25, 1979.

\bibitem{bbi09}
M.~Brin, D.~Burago, and S.~Ivanov.
\newblock Dynamical coherence of partially hyperbolic diffeomorphisms of the 3-torus.
\newblock {\em J. Mod. Dyn.}, 3(1):1--11, 2009.

\bibitem{bw08}
K.~Burns and A.~Wilkinson.
\newblock Dynamical coherence and center bunching.
\newblock {\em Discrete Contin. Dyn. Syst.}, 22(1-2):89--100, 2008.

\bibitem{cp15}
B.~Crovisier and R.~Potire.
\newblock Introduction to partially hyperbolic dynamics.
\newblock School on Dynamical Systems, ICTP, Trieste, 2015.

\bibitem{lla92}
R.~de~la Llave.
\newblock Smooth conjugacy and {S}-{R}-{B} measures for uniformly and non-uniformly hyperbolic systems.
\newblock {\em Comm. Math. Phys.}, 150(2):289--320, 1992.

\bibitem{lmmt2}
R.~de~la Llave, J.~Marco, and R.~Moriy\'on.
\newblock Invariants for smooth conjugacy of hyperbolic dynamical systems. {I-IV}.
\newblock {\em Comm. Math. Phys.}, 109,112,116, 1987-1988.

\bibitem{dew21}
J.~Dewitt.
\newblock Local {L}yapunov spectrum rigidity of nilmanifold automorphisms.
\newblock {\em J. Mod. Dyn.}, 17:65--109, 2021.

\bibitem{dg24}
J.~Dewitt and A.~Gogolev.
\newblock Dominated splitting from constant periodic data and global rigidity of {A}nosov automorphisms.
\newblock {\em Geom. Funct. Anal.}, 2024.

\bibitem{diaz19}
L.~J. D\'{\i}az, K.~Gelfert, M.~Gr\"{o}ger, and T.~J\"{a}ger.
\newblock Hyperbolic graphs: critical regularity and box dimension.
\newblock {\em Trans. Amer. Math. Soc.}, 371(12):8535--8585, 2019.

\bibitem{fal90}
K.~Falconer.
\newblock {\em Fractal geometry}.
\newblock John Wiley \& Sons, Ltd., Chichester, 1990.
\newblock Mathematical foundations and applications.

\bibitem{fal88}
K.~J. Falconer.
\newblock The {H}ausdorff dimension of self-affine fractals.
\newblock {\em Math. Proc. Cambridge Philos. Soc.}, 103(2):339--350, 1988.

\bibitem{F06}
T.~Fisher.
\newblock Thesis.
\newblock {\em Available online at http://www.etda.libraries.psu.edu/theses/approved/WorldWideFiles/ETD-1174/fisher-thesis-draft.pdf. PennState, 2006.}, 2006.

\bibitem{fk91}
L.~Flaminio and A.~Katok.
\newblock Rigidity of symplectic {A}nosov diffeomorphisms on low-dimensional tori.
\newblock {\em Ergodic Theory Dynam. Systems}, 11(3):427--441, 1991.

\bibitem{fh03}
P.~Foulon and B.~Hasselblatt.
\newblock Zygmund strong foliations.
\newblock {\em Israel J. Math.}, 138:157--169, 2003.

\bibitem{fra69}
J.~Franks.
\newblock Anosov diffeomorphisms on tori.
\newblock {\em Trans. Amer. Math. Soc.}, 145:117--124, 1969.

\bibitem{gs20}
S.~Gan and Y.~Shi.
\newblock Rigidity of center {L}yapunov exponents and {$su$}-integrability.
\newblock {\em Comment. Math. Helv.}, 95(3):569--592, 2020.

\bibitem{ghy93}
\'E. Ghys.
\newblock Rigidit\'e{} diff\'erentiable des groupes fuchsiens.
\newblock {\em Inst. Hautes \'Etudes Sci. Publ. Math.}, (78):163--185, 1993.

\bibitem{gog08}
A.~Gogolev.
\newblock Smooth conjugacy of {A}nosov diffeomorphisms on higher-dimensional tori.
\newblock {\em J. Mod. Dyn.}, 2(4):645--700, 2008.

\bibitem{gog2010}
A.~Gogolev.
\newblock Diffeomorphisms {H}{\"o}lder conjugate to anosov diffeomorphisms.
\newblock {\em Ergodic Theory and Dynamical Systems}, 30(2):441--456, 2010.

\bibitem{gog12}
A.~Gogolev.
\newblock How typical are pathological foliations in partially hyperbolic dynamics: an example.
\newblock {\em Israel J. Math.}, 187:493--507, 2012.

\bibitem{gog17}
A.~Gogolev.
\newblock Bootstrap for local rigidity of {A}nosov automorphisms on the 3-torus.
\newblock {\em Comm. Math. Phys.}, 352(2):439--455, 2017.

\bibitem{gg08}
A.~Gogolev and M.~Guysinsky.
\newblock {$C^1$}-differentiable conjugacy of {A}nosov diffeomorphisms on three dimensional torus.
\newblock {\em Discrete Contin. Dyn. Syst.}, 22(1-2):183--200, 2008.

\bibitem{gks11}
A.~Gogolev, B.~Kalinin, and V.~Sadovskaya.
\newblock Local rigidity for {A}nosov automorphisms.
\newblock {\em Math. Res. Lett.}, 18(5):843--858, 2011.
\newblock With an appendix by Rafael de la Llave.

\bibitem{gks20}
A.~Gogolev, B.~Kalinin, and V.~Sadovskaya.
\newblock Local rigidity of {L}yapunov spectrum for toral automorphisms.
\newblock {\em Israel J. Math.}, 238(1):389--403, 2020.

\bibitem{gu23}
R.~Gu.
\newblock Smooth stable foliations of {A}nosov diffeomorphisms.
\newblock {\em arXiv:2310.19088}, 2023.

\bibitem{hnw02}
D.~Hadjiloucas, M.~Nicol, and C.~Walkden.
\newblock Regularity of invariant graphs over hyperbolic systems.
\newblock {\em Ergodic theory and Dynamical systems}, 22(2):469--482, 2002.

\bibitem{hs21}
A.~Hammerlindl and Y.~Shi.
\newblock Accessibility of derived-from-{A}nosov systems.
\newblock {\em Trans. Amer. Math. Soc.}, 374(4):2949--2966, 2021.

\bibitem{hasboot}
B.~Hasselblatt.
\newblock Bootstrapping regularity of the {A}nosov splitting.
\newblock {\em Proc. Amer. Math. Soc.}, 115(3):817--819, 1992.

\bibitem{has95}
B.~Hasselblatt.
\newblock Regularity of the {A}nosov splitting and of horospheric foliations.
\newblock {\em Ergodic Theory Dynam. Systems}, 14(4):645--666, 1994.

\bibitem{has97}
B.~Hasselblatt.
\newblock Regularity of the {A}nosov splitting. {II}.
\newblock {\em Ergodic Theory Dynam. Systems}, 17(1):169--172, 1997.

\bibitem{has02}
B.~Hasselblatt.
\newblock Critical regularity of invariant foliations.
\newblock {\em Discrete Contin. Dyn. Syst.}, 8(4):931--937, 2002.

\bibitem{hw99}
B.~Hasselblatt and A.~Wilkinson.
\newblock Prevalence of non-{L}ipschitz {A}nosov foliations.
\newblock {\em Ergodic Theory Dynam. Systems}, 19(3):643--656, 1999.

\bibitem{ph07}
M.~Hirayama and Y.~Pesin.
\newblock Non-absolutely continuous foliations.
\newblock {\em Israel J. Math.}, 160:173--187, 2007.

\bibitem{hps77}
M.~W. Hirsch, C.~C. Pugh, and M.~Shub.
\newblock {\em Invariant manifolds}, volume Vol. 583 of {\em Lecture Notes in Mathematics}.
\newblock Springer-Verlag, Berlin-New York, 1977.

\bibitem{hoc14}
M.~Hochman.
\newblock On self-similar sets with overlaps and inverse theorems for entropy.
\newblock {\em Ann. of Math. (2)}, 180(2):773--822, 2014.

\bibitem{HR}
M.~Hochman and A.~Rapaport.
\newblock Hausdorff dimension of planar self-affine sets and measures with overlaps.
\newblock {\em J. Eur. Math. Soc. (JEMS)}, 24(7):2361--2441, 2022.

\bibitem{hopf}
E.~Hopf.
\newblock Statistik der geod\"atischen {L}inien in {M}annigfaltigkeiten negativer {K}r\"ummung.
\newblock {\em Ber. Verh. S\"achs. Akad. Wiss. Leipzig Math.-Phys. Kl.}, 91:261--304, 1939.

\bibitem{hl93}
T.~Hu and K.~Lau.
\newblock Fractal dimensions and singularities of the {W}eierstrass type functions.
\newblock {\em Trans. Amer. Math. Soc.}, 335(2):649--665, 1993.

\bibitem{kh90}
S.~Hurder and A.~Katok.
\newblock Differentiability, rigidity and {G}odbillon-{V}ey classes for {A}nosov flows.
\newblock {\em Inst. Hautes \'Etudes Sci. Publ. Math.}, (72):5--61, 1990.

\bibitem{jpl95}
M.~Jiang, Ya.~B. Pesin, and R.~de~la Llave.
\newblock On the integrability of intermediate distributions for {A}nosov diffeomorphisms.
\newblock {\em Ergodic Theory Dynam. Systems}, 15(2):317--331, 1995.

\bibitem{JLPX}
Y.~Jiao, J.~Li, W.~Pan, and D.~Xu.
\newblock On the dimension of limit sets on {$\Bbb{P}(\Bbb{R}^3)$} via stationary measures: variational principles and applications.
\newblock {\em IMRN}, 2024(19):13015--13045, 2024.

\bibitem{jou88}
J.-L. Journ\'e.
\newblock A regularity lemma for functions of several variables.
\newblock {\em Rev. Mat. Iberoamericana}, 4(2):187--193, 1988.

\bibitem{kal11}
B.~Kalinin.
\newblock Liv\v sic theorem for matrix cocycles.
\newblock {\em Ann. of Math. (2)}, 173(2):1025--1042, 2011.

\bibitem{ks09}
B.~Kalinin and V.~Sadovskaya.
\newblock On {A}nosov diffeomorphisms with asymptotically conformal periodic data.
\newblock {\em Ergodic Theory Dynam. Systems}, 29(1):117--136, 2009.

\bibitem{ksw23}
B.~Kalinin, V.~Sadovskaya, and Z.~Wang.
\newblock Smooth local rigidity for hyperbolic toral automorphisms.
\newblock {\em Commun. Am. Math. Soc.}, 3:290--328, 2023.

\bibitem{ksw24}
B.~Kalinin, V.~Sadovskaya, and Z.~Wang.
\newblock Global smooth rigidity for toral automorphisms.
\newblock {\em arXiv:2407.13877}, 2024.

\bibitem{kap84}
J.L. Kaplan, J.~Mallet-Paret, and J.A. Yorke.
\newblock The {L}yapunov dimension of a nowhere differentiable attracting torus.
\newblock {\em Ergodic Theory Dynam. Systems}, 4(2):261--281, 1984.

\bibitem{kh95}
A.~Katok and B.~Hasselblatt.
\newblock {\em Introduction to the modern theory of dynamical systems}, volume~54 of {\em Encyclopedia of Mathematics and its Applications}.
\newblock Cambridge University Press, Cambridge, 1995.

\bibitem{LPX}
J.~Li, W.~Pan, and D.~Xu.
\newblock On the dimension of limit sets on {$\Bbb{P}(\Bbb{R}^3)$} via stationary measures: the theory and applications.
\newblock {\em 2311.10265}, 2023.

\bibitem{man74}
A.~Manning.
\newblock There are no new {A}nosov diffeomorphisms on tori.
\newblock {\em Amer. J. Math.}, 96:422--429, 1974.

\bibitem{mat95}
P.~Mattila.
\newblock {\em Geometry of sets and measures in {E}uclidean spaces: Fractals and rectifiability}, volume~44 of {\em Cambridge Studies in Advanced Mathematics}.
\newblock Cambridge University Press, Cambridge, 1995.

\bibitem{mil97}
J.~Milnor.
\newblock Fubini foiled: {K}atok’s paradoxical example in measure theory.
\newblock {\em The Mathematical Intelligencer}, 19(2):30--32, 1997.

\bibitem{por15}
R.~Potrie.
\newblock Partial hyperbolicity and foliations in {$\Bbb{T}^3$}.
\newblock {\em J. Mod. Dyn.}, 9:81--121, 2015.

\bibitem{psw97}
C.~Pugh, M.~Shub, and A.~Wilkinson.
\newblock H\"{o}lder foliations.
\newblock {\em Duke Math. J.}, 86(3):517--546, 1997.

\bibitem{psw10}
C.~Pugh, M.~Shub, and A.~Wilkinson.
\newblock H\"older foliations, revisited.
\newblock {\em J. Mod. Dyn.}, 6(1):79--120, 2012.

\bibitem{R}
A.~Rapaport.
\newblock On self-affine measures associated to strongly irreducible and proximal systems.
\newblock {\em Adv. Math.}, 449:Paper No. 109734, 116, 2024.

\bibitem{rs21}
H.~Ren and W.~Shen.
\newblock A dichotomy for the {W}eierstrass-type functions.
\newblock {\em Invent. Math.}, 226(3):1057--1100, 2021.

\bibitem{rgz17}
Y.~Ren, S.~Gan, and P.~Zhang.
\newblock Accessibility and homology bounded strong unstable foliation for {A}nosov diffeomorphisms on 3-torus.
\newblock {\em Acta Math. Sin. (Engl. Ser.)}, 33(1):71--76, 2017.

\bibitem{hertznote}
F.~Rodriguez~Hertz, M.~Rodriguez~Hertz, and R.~Ures.
\newblock A survey of partially hyperbolic dynamics.
\newblock In {\em Partially hyperbolic dynamics, laminations, and {T}eichm\"uller flow}, volume~51 of {\em Fields Inst. Commun.}, pages 35--87. Amer. Math. Soc., Providence, RI, 2007.

\bibitem{sx09}
R.~Saghin and Z.~Xia.
\newblock Geometric expansion, {L}yapunov exponents and foliations.
\newblock {\em Ann. Inst. H. Poincar\'e{} C Anal. Non Lin\'eaire}, 26(2):689--704, 2009.

\bibitem{sy19}
R.~Saghin and J.~Yang.
\newblock Lyapunov exponents and rigidity of {A}nosov automorphisms and skew products.
\newblock {\em Adv. Math.}, 355:106764, 45, 2019.

\bibitem{shen18}
W.~Shen.
\newblock Hausdorff dimension of the graphs of the classical {W}eierstrass functions.
\newblock {\em Math. Z.}, 289(1-2):223--266, 2018.

\bibitem{sw00}
M.~Shub and A.~Wilkinson.
\newblock Pathological foliations and removable zero exponents.
\newblock {\em Invent. Math.}, 139(3):495--508, 2000.

\bibitem{var16}
R.~Var\~ao.
\newblock Center foliation: absolute continuity, disintegration and rigidity.
\newblock {\em Ergodic Theory Dynam. Systems}, 36(1):256--275, 2016.

\end{thebibliography}
\end{document}